\documentclass[12pt]{article}
\usepackage{a4wide}
\usepackage{amsmath, amsthm, amsfonts, amssymb, bbm}
\usepackage{graphics}
\usepackage{xypic}
\usepackage[all]{xy}
\usepackage[english]{babel}
\usepackage[font=small,format=plain,labelfont=bf,up]{caption}
\usepackage{hyperref}

\allowdisplaybreaks

\theoremstyle{plain}

\newtheorem{theorem}{Theorem}
\newtheorem{lemma}[theorem]{Lemma}
\newtheorem{proposition}[theorem]{Proposition}
\newtheorem{corollary}[theorem]{Corollary}
\newtheorem{conjecture}[theorem]{Conjecture}

\theoremstyle{definition}

\newtheorem{remark}[theorem]{Remark}
\newtheorem{definition}[theorem]{Definition}

 \numberwithin{equation}{section}
 \numberwithin{theorem}{section}
 
 \newcommand\ipic[1]		{\raisebox{-0.5\height}{\scalebox{.45}{\includegraphics{pic#1.eps}}}}

\newcommand{\bth}{\begin{theorem}}
\renewcommand{\eth}{\end{theorem}}
\newcommand{\bpr}{\begin{proposition}}
\newcommand{\epr}{\end{proposition}}
\newcommand{\ble}{\begin{lemma}}
\newcommand{\ele}{\end{lemma}}
\newcommand{\bco}{\begin{corollary}}
\newcommand{\eco}{\end{corollary}}
\newcommand{\bde}{\begin{definition}}
\newcommand{\ede}{\end{definition}}
\newcommand{\bex}{\begin{example}}
\newcommand{\eex}{\end{example}}
\newcommand{\bre}{\begin{remark}}
\newcommand{\ere}{\end{remark}}

\DeclareMathOperator{\End}{End}
\DeclareMathOperator{\Hom}{Hom}
\DeclareMathOperator{\Rep}{Rep}

\DeclareMathOperator{\coev}{coev}

\newcommand\be            {\begin{equation}}
\newcommand\ee            {\end{equation}}

\newcommand{\ot}{\otimes}
\newcommand{\lb}{\label}

\newcommand{\VOA}{\mathbb{V}}

\newcommand\vect{\mathcal{V}\hspace{-.5pt}ect}
\newcommand\svect{\mathcal{S}\mathcal{V}\hspace{-.5pt}ect}
\newcommand\SF{\mathcal{S}\hspace{-.65pt}\mathcal{F}}

\newcommand\eps           {\varepsilon}
\newcommand\id            {id}

\newcommand\one           {{\bf1}}

\newcommand\Cb            {\mathbb{C}}

\newcommand\Rb            {\mathbb{R}}
\newcommand\Zb            {\mathbb{Z}}

\newcommand\Vc            {\mathcal{V}}

\newcommand\h            {\mathfrak{h}}
\newcommand\f	{\mathfrak{f}}

\newcommand{\Lmul}{\hspace{2.5pt}\begin{picture}(-1,1)(-1,-3)\circle*{3}\end{picture}\hspace{5.5pt}}
    
\newcommand\void[1]	{}

\newcommand{\Vect}{{\cal V}{\it ect}}
\newcommand{\sVect}{{\it s}{\cal V}{\it ect}}
\newcommand{\C}{{\cal C}}

\newcommand{\bpf}{\begin{proof}}
\newcommand{\epf}{\end{proof}}

\begin{document}

\thispagestyle{empty}
\def\thefootnote{\fnsymbol{footnote}}
\begin{flushright}
ZMP-HH/16-2\\
Hamburger Beitr\"age zur Mathematik 580
\end{flushright}
\vskip 3em
\begin{center}\LARGE
Holomorphic Symplectic Fermions
\end{center}

\vskip 2em
\begin{center}
{\large 
Alexei Davydov$^{a}$,\, Ingo Runkel$^{b}$,\, ~\footnote{Emails: {\tt davydov@ohio.edu}, {\tt ingo.runkel@uni-hamburg.de}}}
\\[1em]
\it$^a$ 
Department of Mathematics, Ohio University, \\ Athens, OH 45701, USA
\\[1em]
$^b$ Fachbereich Mathematik, Universit\"at Hamburg\\
Bundesstra\ss e 55, 20146 Hamburg, Germany
\end{center}

\vskip 2em
\begin{center}
  January 2016
\end{center}
\vskip 2em

\begin{abstract}
Let $V$ be the even part of the vertex operator super-algebra of $r$ pairs of symplectic fermions. Up to two conjectures, we show that $V$ admits a unique holomorphic extension if $r$ is a multiple of $8$, and no holomorphic extension otherwise.

This is implied by two results obtained in this paper: 
1) If $r$ is a multiple of 8, one possible holomorphic extension is given by the lattice vertex operator algebra for the even self dual lattice $D_{r}^+$ with shifted stress tensor. 2) We classify  Lagrangian algebras in $\SF(\h)$, a ribbon category associated to symplectic fermions. 

The classification of holomorphic extensions of $V$ follows from 1) and 2) if one assumes that $\SF(\h)$ is ribbon equivalent to $\mathrm{Rep}(V)$, and that simple modules of extensions of $V$ are in one-to-one relation with simple local modules of the corresponding commutative algebra in $\SF(\h)$.
\end{abstract}

\setcounter{footnote}{0}
\def\thefootnote{\arabic{footnote}}

\newpage

\tableofcontents

\section{Introduction and summary}

The chiral conformal field theory of symplectic fermions was introduced in \cite{Kausch:1995py}. It was realised in \cite{Gaberdiel:1998ps} that despite the appearance of logarithms in the conformal blocks, one can construct a full (that is, non-chiral) conformal field theory from symplectic fermions. By now, symplectic fermions are the best studied example of a (chiral and non-chiral) logarithmic conformal field theory, see e.g.\ \cite{
	Gaberdiel:1996np,
	Fuchs:2003yu,
	Abe:2005,
	Gaberdiel:2006pp,
	Nagatomo:2009xp,
	Abe:2011ab,
	Gainutdinov:2011ab,
	Creutzig:2011cu,
	Tsuchiya:2012ru,
	Gainutdinov:2015lja} and references therein.

To define symplectic fermions, one starts from a finite-dimensional symplectic vector space $\h$ over $\Cb$. Up to isomorphism, $\h$ is characterised by its dimension $d \in 2\Zb_{>0}$. From $\h$ one constructs the affine Lie super-algebra $\widehat \h$ of purely odd free super-bosons (Section \ref{sec:SF-VOSA}). 
The vacuum $\widehat\h$-module, that is, the highest weight $\widehat\h$-module of highest weight $0$, carries the structure of a vertex operator super-algebra $\VOA(d)$ (VOSA) of central charge $-d$ \cite{Abe:2005}. This is the symplectic fermion VOSA. Its even subspace $\VOA(d)_\mathrm{ev}$ is
a vertex operator algebra (VOA). 

The characters of the irreducible representations of $\VOA(d)_\mathrm{ev}$ and their modular properties are known \cite{Kausch:1995py,Gaberdiel:1996np,Abe:2005}. Independently of $d$, there are four irreducible representations. Let us denote their characters by $\chi_i(\tau)$, $i=1,\dots,4$, so that $\chi_1(\tau)$ is the character of $\VOA(d)_\mathrm{ev}$, see Section \ref{sec:char-mod-inv} for details. One can ask if there are non-zero linear combinations $Z(\tau) = \sum_{i=1}^4 z_i \,\chi_i(\tau)$ with $z_i \in \Zb_{\ge 0}$ such that $Z(\tau)$ is modular invariant or at least projectively modular invariant. By the latter we mean that
$Z(-\frac{1}{\tau}) = \xi\,Z(\tau)$ and $Z(\tau+1) = \zeta \, Z(\tau)$ for some $\xi,\zeta \in \mathbb{C}^\times$.
An easy computation shows (Proposition \ref{prop:(projectively)modinv}) that such $Z(\tau)$ exist if and only if $d \in 16\Zb_{>0}$, in which case
\be\label{eq:all-mod-inv-intro}
	(z_1,z_2,z_3,z_4) \in \Zb_{>0} \, (2^{\frac{d}2-1} , 2^{\frac{d}2-1} , 1 , 0 ) \ .
\ee
The resulting projective modular invariant $Z(\tau)$ is modular invariant iff $d \in 48 \Zb_{>0}$. 

A {\em holomorphic VOA} is a VOA $\VOA$ all of whose simple modules are isomorphic to $\VOA$. If $\VOA$ has no negative conformal weights\footnote{This condition actually does not hold in the case we are interested in. However, one can find a modified stress tensor with respect to which all conformal weights are non-negative.\label{fn:non-neg-wt}} and is $C_2$-cofinite, being holomorphic implies that the character of $\VOA$ is modular invariant \cite{Zhu1996,Miyamoto:2002ar}. 

By a {\em holomorphic extension} of a VOA $\mathbb{U}$ we mean a conformal embedding of $\mathbb{U}$ into holomorphic VOA $\VOA$. The above observation on characters hence raises the following two questions:
\begin{itemize}
\item[Q1)] Is the minimal solution in \eqref{eq:all-mod-inv-intro}, that is, $Z_\mathrm{min}(\tau) =  2^{\frac{d}2-1} \big\{ \chi_1(\tau) + \chi_2(\tau) \big\} + \chi_3(\tau)$, the character of a holomorphic extension of $\VOA(d)_\mathrm{ev}$?
\item[Q2)] What are all holomorphic extensions of $\VOA(d)_\mathrm{ev}$?
\end{itemize}
Question 1 turns out to be easy to answer in the affirmative if one combines two observations. Firstly, $Z_\mathrm{min}(\tau)$ is the character of the lattice VOA $\VOA_{D_{d/2}^+}$ of the even self-dual lattice $D_{d/2}^+$ (Section \ref{sec:char-mod-inv}). Secondly, the VOSA $\VOA(d)$ is a sub-VOSA of the lattice VOSA $\VOA_{\Zb^{d/2}}$ of the lattice $\Zb^{d/2}$ with its standard inner product (see Theorem \ref{thm:symp-ferm-in-lattice} -- this holds for all $d \in \Zb_{>0}$). However, $\VOA_{\Zb^{d/2}}$ does not have the standard stress tensor (aka.\ Virasoro element or conformal vector), but instead its stress tensor is
\be\label{eq:T-FF_for_Zd2}
	T^\mathrm{FF} = \tfrac12 \sum_{i=1}^{d/2} \big( H^i_{-1}H^i_{-1} -H^i_{-2} \big) \one
\ee
(`FF' for Feigin-Fuchs) with central charge $-d$, where $H^i_m$, $i=1,\dots,d/2$, $m\in\Zb$ are the generators of the $d/2$ copies of the Heisenberg algebra (Section \ref{sec:SF-VOSA}).
Free boson theories with modified stress tensor as above appear in the conformal field theory literature as ``free boson with background charge'' or ``Feigin-Fuchs free boson'', see e.g.\ \cite{DMS}.
The construction also appears in the context of VOAs, see e.g. \cite{Kac:1998}. A detailed investigation of lattice VOAs with such a shifted stress tensor can be found in \cite{Dong:2006}. 

Combining the two observations above gives our first main result (Corollary \ref{cor:symp-even-in-lattice}):

\begin{theorem}\label{thm:intro-main-1}
For $d \in 16 \Zb_{>0}$, $\VOA(d)_\mathrm{ev}$ has a holomorphic extension, namely it is a sub-VOA of the holomorphic lattice VOA $\VOA_{D_{d/2}^+}$ with stress tensor \eqref{eq:T-FF_for_Zd2}.
\end{theorem}

The character of $\VOA_{D_{d/2}^+}$ is precisely $Z_\mathrm{min}(\tau)$, and so this theorem answers question Q1.
Question Q2 is more involved, and we can only give an answer under an additional assumption and a conjectural equivalence of categories. Let us present this result in more detail. 

\medskip

In addition to the affine Lie super-algebra $\widehat\h$, we will need a twisted version thereof, $\widehat\h_{\mathrm{tw}}$. The category of $\widehat\h$ and $\widehat\h_{\mathrm{tw}}$ representations, subject to certain finiteness properties, is equivalent to
\be
	\SF(\h) \,:=\, \Rep^\mathrm{fd}(\h) \oplus  \svect^\mathrm{fd}
\ee
as a $\Cb$-linear category (see \cite{Runkel:2012cf} and Section \ref{sec:brmon-on-rep}). Here, $\svect^\mathrm{fd}$ is the category of finite-dimensional super-vector spaces, and $\Rep^\mathrm{fd}(\h)$ is the category of representations of $\h$ -- thought of as an abelian Lie super-algebra -- in $\svect^\mathrm{fd}$. 
We remark that $\Rep^\mathrm{fd}(\h)$ is not semi-simple.

For an object $A \in \SF(\h)$ we will denote by 
\be
\widehat{A}
\ee 
the corresponding direct sum of an $\widehat\h$ and an $\widehat\h_{\mathrm{tw}}$ representation, and by $(\widehat A)_\mathrm{ev}$ its parity-even subspace. In \cite{Runkel:2012cf}, the monodromy and asymptotic properties of conformal blocks were used to endow $\SF(\h)$ with the structure of a braided tensor category. Altogether, 
the category $\SF(\h)$
\begin{itemize} \itemsep -.1em
\item[-] is $\Cb$-linear and abelian, 
\item[-] has four simple objects (independently of $\h$), but is not semi-simple (unless $\h=0$, which we exclude),
\item[-] is finite (all Hom-spaces are finite-dimensional and all objects have finite length), and 
\item[-] is equipped with a ribbon structure.
\end{itemize}
The associator, braiding and ribbon twist of $\SF(\h)$ depend on an additional parameter $\beta$ subject to $\beta^4 = (-1)^{d/2}$, which is omitted from the notation. In the case of the symplectic fermions, $\beta$ takes the value 
\be\label{eq:beta-sympferm-value}
	\beta = e^{- \pi i d/8} \ .
\ee	

A {\em Lagrangian algebra} in an abelian ribbon category is a commutative associative unital algebra with trivial ribbon twist, whose only local modules are the algebra itself or direct sums thereof \cite{Frohlich:2003hm,dmno}.
The definition of local modules is recalled in Section \ref{sec:commalg} and that of Lagrangian algebras in Section \ref{sec:locmod}.

The second main result of this paper is
 (Theorem \ref{thm:Lag-in-SF}):

\begin{theorem}\label{thm:intro-main-2}
\begin{enumerate}
\item $\SF(\h)$ contains Lagrangian algebras iff $\beta=1$.
\item For $\beta=1$, Lagrangian algebras are determined up to isomorphism of algebras by a Lagrangian subspace $\f\subset\h$ -- denote this algebra by $H(\f)$.
\item For any two Lagrangian subspaces $\f,\f' \subset \h$ there is a 
	$\Cb$-linear
ribbon autoequivalence $J$ of $\SF(\h)$ such that $J(H(\f)) \cong H(\f')$ as algebras.
\end{enumerate}
\end{theorem}

Note that the condition $\beta=1$ requires $d \in 4 \Zb$, and in the case of symplectic fermions $d \in 16\Zb$. 
In the symplectic fermion case one furthermore verifies that when $H(\f)$ exists, i.e.\ when $d \in 16\mathbb{Z}$, the character of $(\widehat{H(\f)})_\mathrm{ev}$ is $Z_\mathrm{min}(\tau)$.

\medskip

Theorem \ref{thm:intro-main-2} is purely a result about the braided tensor and ribbon structure of $\SF(\h)$. 
To obtain statements about the representation category $\Rep(\VOA(d)_\mathrm{ev})$ of the VOA $\VOA(d)_\mathrm{ev}$, we recall that
$\VOA(d)_\mathrm{ev}$ is $C_2$-cofinite  \cite{Abe:2005} and thus according to \cite{Huang:2010}, $\Rep(\VOA(d)_\mathrm{ev})$ is a braided tensor category.
We conjecture (see Conjecture \ref{conj:braided-equiv} for a more detailed version):

\begin{conjecture}\label{conj:braided-equiv-intro}
For $\beta = e^{- \pi i d/8}$, the functor
\be
F ~:~ \SF(\h)\longrightarrow\Rep(\VOA(d)_\mathrm{ev})\quad ,\quad X\longmapsto ({\widehat X})_\mathrm{ev} \ ,
\ee 
is a $\Cb$-linear braided monoidal equivalence. 
\end{conjecture}

In \cite{Huang:2014ixa} it is shown that for simple $C_2$-cofinite VOAs $\VOA$, there is a one-to-one correspondence between commutative algebras $A$ with trivial twist in $\Rep(\VOA)$ and extensions $\VOA_A$ of $\VOA$. 
If $\VOA$ is rational and satisfies some additional technical conditions, it is furthermore established in \cite{Huang:2014ixa} that local $A$-modules are in one-to-one correspondence to $\VOA_A$-modules.

We can combine the results of Theorem \ref{thm:intro-main-1} and \ref{thm:intro-main-2} as follows (see Theorem \ref{thm:classify-extensions}):

\begin{theorem}\label{thm:intro-main-3}
Assume Conjecture \ref{conj:braided-equiv-intro} and that the equivalence between $\VOA_A$-modules and local $A$-modules also holds for $\VOA(d)_\mathrm{ev}$. Then $\VOA(d)_\mathrm{ev}$ has a holomorphic extension iff $d \in 16 \Zb_{>0}$. In this case, every holomorphic extension of $\VOA(d)_\mathrm{ev}$ is isomorphic as a VOA to $\VOA_{D_{d/2}^+}$ with stress tensor \eqref{eq:T-FF_for_Zd2}.
\end{theorem}

Modulo the two assumptions we had to make, this answers question Q2.

\medskip

As a consistency check, in Appendix \ref{sec:OPEs} we analyse the vertex operator corresponding to $H(\f)$ and show that it does indeed contain a rank-$d/2$ Heisenberg VOA.

\bigskip
\noindent
{\bf Acknowledgements:}
The authors would like to thank 
Matthias Gaberdiel,
Terry Gannon,
Simon Lentner,
Geoffrey Mason
and
G\'erard Watts
for stimulating discussions.
These results were presented by IR at the conference
``The mathematics of conformal field theory'' (Australian National University, July 2015)
and at the
``Conference on Lie algebras, vertex operator algebras, and related topics'' (Notre Dame, August 2015).
IR would like to thank the organisers of both conferences for very stimulating meetings.

\section{Characters and modular invariants}\label{sec:char-mod-inv}

\subsection{Characters}\label{sec:symp-ferm-char}

The starting point of the study in this paper is a simple observation on the characters of symplectic fermions. Let $\VOA(d)$ the  vertex operator super-algebra (VOSA) of $d \in 2\Zb_{>0}$ symplectic fermions (we will define it in detail in Section \ref{sec:SF-VOSA}). It has central charge $c=-d$. We write 
\begin{equation}\label{eq:V(d)=V0+V1}
\VOA(d) ~=~ \VOA(d)_\mathrm{ev} \oplus \VOA(d)_\mathrm{odd}
\end{equation} 
for its decomposition into the parity even and odd part.  $\VOA(d)_\mathrm{ev}$ has four irreducible representations \cite{Kausch:1995py,Gaberdiel:1996np,Abe:2005}. Their lowest conformal weights are
\be\label{eq:V0-irrep-weights}
  h_1 = 0 \quad , \qquad
  h_2 = 1 \quad , \qquad
  h_3 = -\tfrac{d}{16} \quad , \qquad
  h_4 = -\tfrac{d}{16} + \tfrac12 \ .
\ee
To express their characters, we first define the functions (as usual, $q=e^{2 \pi i \tau}$)
\begin{align}
\chi_{ns,+}(\tau) &=  \Big( q^{\frac1{24}} \prod_{n=1}^\infty (1+q^n) \Big)^d \ ,
\qquad &
\chi_{ns,-}(\tau) &=  \Big( q^{\frac1{24}} \prod_{n=1}^\infty (1-q^n) \Big)^d\ ,
\nonumber \\
\chi_{r,+}(\tau) &=  \Big( q^{-\frac1{48}} \prod_{n=1}^\infty (1+q^{n-\frac12}) \Big)^d\ ,
\qquad &
\chi_{r,-}(\tau) &=  \Big( q^{-\frac1{48}} \prod_{n=1}^\infty (1-q^{n-\frac12}) \Big)^d\ .
\end{align}
Here, $\chi_{ns,\pm}$ is the character of $\VOA(d)$ with or without insertion of the parity involution $\omega$, and $\chi_{r,\pm}$ are characters of a twisted $\VOA(d)$-module.
The characters of the four irreducibles on $\VOA(d)_\mathrm{ev}$ are \cite{Kausch:1995py,Gaberdiel:1996np,Abe:2005}, in the order stated in \eqref{eq:V0-irrep-weights},
\begin{align}
  \chi_1(\tau) &= \tfrac12\big( \chi_{ns,+}(\tau) + \chi_{ns,-}(\tau) \big) \ ,
\qquad &
  \chi_2(\tau) &= \tfrac12\big( \chi_{ns,+}(\tau) - \chi_{ns,-}(\tau) \big)\ ,
\nonumber \\
  \chi_3(\tau) &= \tfrac12\big( \chi_{r,+}(\tau) + \chi_{r,-}(\tau) \big)\ ,
\qquad &
  \chi_4(\tau) &= \tfrac12\big( \chi_{r,+}(\tau) - \chi_{r,-}(\tau) \big) \ .
\end{align}
Their behaviour under S and T modular transformations is \cite{Gaberdiel:1996np,Abe:2005}: \begin{align}
\chi_1(\tau{+}1) &= e^{\frac{\pi i d}{12} } \, \chi_1(\tau)
~,~&
\chi_1\big(\tfrac{-1}{\tau}\big) &= 2^{-\frac{d}{2}-1} \big(\chi_3(\tau) - \chi_4(\tau)\big) + \tfrac12 (-i\tau)^\frac{d}{2} \big(\chi_1(\tau) - \chi_2(\tau)\big) \ ,
\nonumber\\
\chi_2(\tau{+}1) &= e^{\frac{\pi i d}{12} } \,\chi_2(\tau) 
~,~&
\chi_2\big(\tfrac{-1}{\tau}\big) &= 2^{-\frac{d}{2}-1} \big(\chi_3(\tau) - \chi_4(\tau)\big) - \tfrac12 (-i\tau)^\frac{d}{2} \big(\chi_1(\tau) - \chi_2(\tau)\big) \ ,
\nonumber\\
\chi_3(\tau{+}1) &= e^{-\frac{\pi i d}{24} } \,\chi_3(\tau)
~,~&
\chi_3\big(\tfrac{-1}{\tau}\big) &= \tfrac{1}{2} \big(\chi_3(\tau) + \chi_4(\tau)\big) + 2^{\frac{d}{2}-1} \big(\chi_1(\tau) + \chi_2(\tau)\big) \ ,
\nonumber\\
\chi_4(\tau{+}1) &= -e^{-\frac{\pi i d}{24} } \,\chi_4(\tau)
~,~&
\chi_4\big(\tfrac{-1}{\tau}\big) &= \tfrac{1}{2} \big(\chi_3(\tau) + \chi_4(\tau)\big) - 2^{\frac{d}{2}-1} \big(\chi_1(\tau) + \chi_2(\tau)\big) \ .
\label{eq:ST-transf}
\end{align}

\subsection{Modular invariant linear combinations of characters}

Define $Z(\tau) = \sum_{i=1}^4 z_i \chi_i(\tau)$ for $z_i \in \Cb$. We would like to know for which subspace of $\Cb^4$ the function $Z(\tau)$ is modular invariant, or at least projectively modular invariant (see below). One can then easily read off all non-negative integer solutions.

\medskip

We say that $Z(\tau)$ is {\em projectively modular invariant} if it satisfies S- and T-invariance up to a factor, $Z(-1/\tau) = \xi \,Z(\tau)$, $Z(\tau+1) = \zeta \,Z(\tau)$ for some $\xi,\zeta \in \Cb^\times$. Since the S-transformation is an involution, a non-zero solution must have $\xi^2=1$.

One quickly checks from \eqref{eq:ST-transf} that the $\Cb$-vector space of solutions to the projective S-invariance condition is spanned by
\be
\begin{cases}
  (0,0,1,1)
  ~,~~  (2^{\frac{d}2-1},2^{\frac{d}2-1},1,0) & ;~\xi = 1 \\
 (2^{\frac{d}2},2^{\frac{d}2},-1,1) & ;~\xi=-1 
\end{cases}   
\qquad .
\ee
The T-transformation in  \eqref{eq:ST-transf} shows that $\chi_3$ and $\chi_4$ cannot appear at the same time in a projectively modular invariant linear combination of characters. This rules out the case $\xi=-1$. For $\xi=1$, the vector $(z_1,z_2,z_3,z_4)$ has to be in one of the subspaces
\be
 \mathrm{a)}
 ~~\Cb \, ( 2^{\frac{d}2-1} , 2^{\frac{d}2-1} , 1 , 0 )
 \quad , \quad
 \mathrm{b)}
 ~~ \Cb \, ( -2^{\frac{d}2-1} , -2^{\frac{d}2-1} , 0 , 1 )
\ee
If $z_3 \neq 0$, T-invariance up to a factor requires $d \in 16 \Zb$, and if $z_4 \neq 0$ one needs $d \in 16 \Zb + 8$. However, only case a) contains non-negative integer solutions which are non-zero. We have shown:

\begin{proposition}\label{prop:(projectively)modinv}
The function $Z(\tau) = \sum_{i=1}^4 z_i\, \chi_i(\tau)$ is non-zero and projectively modular invariant (resp.\ modular invariant) with non-negative integers $z_i$ if and only if $d \in 16 \Zb$ (resp.\ $d \in 48\Zb$) and $(z_1,z_2,z_3,z_4) \in \Zb_{>0} \, (2^{\frac{d}2-1} , 2^{\frac{d}2-1} , 1 , 0 )$. 
\end{proposition}

We will write $Z_\mathrm{min}(\tau)$ for the minimal solution, that is,
\be\label{eq:minimal-Zd}
	Z_\mathrm{min}(\tau) ~=~ 2^{\frac{d}2-1}\big\{ \chi_1(\tau) + \chi_2(\tau) \big\} \,+\, \chi_3(\tau) 
	\qquad , \quad \text{where} ~~ d \in 16 \Zb \ .
\ee
Note that its lowest conformal weight occurs in $\chi_3$ and is $-d/16$.

\subsection[Relation to characters of affine $so(d)$ at level 1]{Relation to characters of affine \boldmath$so(d)$ at level 1}\label{sec:char-rel-so(d)1}

The affine Lie algebra $\widehat{so}(d)_1$, for even $d$, has four irreducible integrable highest weight representations, which we will denote by $\one$, $v$, $s$, $c$ (vacuum, vector, spinor, and conjugate spinor), see e.g.\ \cite[Sect.\,15.5.4]{DMS}. Via the Sugawara constructions, these carry a representation of the Virasoro algebra of central charge $c=d/2$, and their conformal weights are
\be\label{eq:so(d)-irrep-weights}
  h_\one = 0 \quad , \qquad
  h_v = \tfrac12 \quad , \qquad
  h_s = h_c = \tfrac{d}{16}  \ .
\ee
The characters of these representations can be expressed in terms of Jacobi theta functions as
\begin{align}
&\chi_\one^{so}(\tau) = \frac12\Bigg( 
	\Big(\frac{\theta_3(q)}{\eta(q)}\Big)^{\!\!\frac d2} + 
	\Big(\frac{\theta_4(q)}{\eta(q)}\Big)^{\!\!\frac d2} \Bigg)
~~ , \quad
\chi_v^{so}(\tau) = \frac12\Bigg( 
	\Big(\frac{\theta_3(q)}{\eta(q)}\Big)^{\!\!\frac d2} -
	\Big(\frac{\theta_4(q)}{\eta(q)}\Big)^{\!\!\frac d2} \Bigg) \ ,
	\nonumber \\
&\chi_s^{so}(\tau) = \chi_c^{so}(\tau) =  \frac12 \Big(\frac{\theta_2(q)}{\eta(q)}\Big)^{\!\!\frac d2} \ .
\label{eq:so(d)-irrep-char}
\end{align}
Here, the Dedekind eta function and the Jacobi theta functions are given by
\begin{align}
\eta(q) &=  q^{\frac1{24}}\prod_{n=1}^\infty(1-q^n) 
= q^{\frac1{24}}(1 + q + 2q^2 + 3q^3 + \dots)^{-1}
\nonumber\\
\theta_2(q) &= 2q^{\frac1{12}}\, \eta(q) \prod_{n=1}^\infty \big(1+q^{n}\big)^2 =  \sum_{m=-\infty}^\infty q^{\frac12 (m+\frac{1}{2})^2}  =  2q^{\frac1{8}}(1 + q + q^3 + \dots) 
\nonumber\\
\theta_3(q) &=  q^{-\frac1{24}}  \,\eta(q) \prod_{n=1}^\infty \big(1+q^{n-\frac12}\big)^2=  \sum_{m=-\infty}^\infty q^{\frac12  m^2}  =  1 + 2q^{\frac12} + 2q^2 + \dots
\nonumber\\
\theta_4(q) &= q^{-\frac1{24}} \,\eta(q)  \prod_{n=1}^\infty \big(1-q^{n-\frac12}\big)^2 =  \sum_{m=-\infty}^\infty (-1)^m q^{\frac12  m^2}  =  1 - 2q^{\frac12} + 2q^2 + \dots
\end{align}
{}From these formulas we can read off\footnote{
	We thank Matthias Gaberdiel for making us aware of the relation between the modular invariants of Proposition \ref{prop:(projectively)modinv} and the characters of $\widehat{so}(d)_1$.}
\be
	2^{\frac{d}2-1}(\chi_1+\chi_2) = \chi_{s/c}^{so}
	\quad , \quad
	\chi_3 = \chi_\one^{so} \ .
\ee
Thus $Z_\mathrm{min}(\tau) = \chi_\one^{so}(\tau) + \chi_{s/c}^{so}(\tau)$.

In fact,
$Z_\mathrm{min}(\tau)$ is $\eta^{-d/2}$ times the theta series for the even selfdual lattice $D_{d/2}^+$. 
The lattice $D_{r}^+$ is an extension of the 
$so(2r)$ root lattice (see e.g.\ \cite[Sect.\,4.7]{SPLAG}). In more detail, the root lattice of $so(2r)$ is
\be\label{eq:D-root-lattice}
D_{r} = \big\{ v \in \Zb^{r}\,\big|\, {\textstyle \sum_{i=1}^{r} v_i \in 2\Zb} \big\} \ .
\ee
For $r \in 8 \Zb$, we define $D_{r}^+$ 
by setting, for $[1] = (\frac12,\frac12,\dots,\frac12)$,
\be\label{eq:D+-def}
	D_{r}^+ \,=\, D_{r} \cup \big(D_{r} {+} [1]\big) 
	\,=\,
	\big\{ v \in \Zb^{r} \cup \big(\Zb^{r} {+} [1]\big) \,\big|\, {\textstyle \sum_{i=1}^{r} v_i \in 2\Zb } \big\} \ .
\ee
The theta series of $D_{r}^+$ has the presentation \cite{SPLAG}
\be\sum_{v \in D_{r}^+} q^{\frac12(v,v)} = \frac{1}{2}\Big(\theta_2(q)^{r} + \theta_3(q)^{r} + \theta_4(q)^{r} \Big) \ .
\ee
Altogether, for $Z_\mathrm{min}(\tau)$ this gives
\be
	Z_\mathrm{min}(\tau) 
	\,=\, \frac{1}{\eta(q)^{\frac{d}{2}}} \sum_{v \in D_{d/2}^+} q^{\frac12(v,v)}
	\,=\, q^{-\frac{d}{48}}(1 +
	\tfrac{d(d-1)}2 \, q
	+ \dots) \,+\, q^{\frac{d}{24}}2^{\frac{d}{2}-1}(1 + dq + \dots) 
\ .
\ee

\section{Symplectic fermion vertex operator super-algebra}\label{sec:SF-VOSA}

In this section we show that the relation between projectively modular invariant linear combinations of characters of $\VOA(d)_{\mathrm{ev}}$ (Proposition \ref{prop:(projectively)modinv}) and the character of the lattice VOA for the even self-dual lattice $D^+_{d/2}$ has a simple explanation: $\VOA(d)_{\mathrm{ev}}$ can be embedded, as a VOA, in the shifted
lattice VOA for $D^+_{d/2}$ (Corollary \ref{cor:symp-even-in-lattice}).

\subsection{Purely odd free super-bosons}\label{sec:purely-odd-free-super}

The VOSA of $d$ symplectic fermions is defined in \cite{Abe:2005} and is a special case of free super-bosons \cite{Kac:1998,Runkel:2012cf}. We briefly summarise the construction.

Let $\h$ be a finite-dimensional purely odd complex super-vector space 
of dimension
\be
	d = \dim \h ~ \in \, 2 \Zb_{>0} 
\ee
with super-symmetric non-degenerate bilinear form $(-,-)$. That is, for $a,b \in \h$ we have $(a,b) = -(b,a)$. Define the affine Lie super-algebra
\be
	\widehat\h = \h[t,t^{-1}] \oplus \Cb K \ ,
\ee
where the $K$ is even and all $a \otimes t^m$ for $a \in \h$ are odd. We abbreviate $a_m = a \otimes t^m$. The Lie bracket is defined by taking $K$ to be central and, for $a,b\in\h$, $m,n \in \Zb$,
\be\label{eq:SF-liebracket}
	[a_m,b_n] = (a,b) \,m \, \delta_{m+n,0}\, K \ .
\ee
Here we use the convention of Lie super-algebras that $[a_m,b_n]$ is symmetric since $a_m,b_n$ are both odd. 

Write $\widehat\h_{\ge 0} \subset \widehat\h$ for the sub Lie super-algebra spanned by the $a_m$ with $a \in \h$ and $m \ge 0$. Let $\Cb_1$ be the $\widehat\h_{\ge 0} \oplus \Cb K$-module on which $K$ acts as $1$ and $\widehat\h_{\ge 0}$ acts as zero. As a $\widehat\h$-module, $\VOA(\h)$ is the induced module
\be\label{eq:free-superbos-VOSA}
	\VOA(\h) = U(\widehat\h) \otimes_{\widehat\h_{\ge 0} \oplus \Cb K} \Cb_1 \ .
\ee
Note that $\VOA(\h)$ is in particular a super-vector space.
On $\VOA(\h)$, for each $a \in \h$ we get a field
\be
	a(x) = \sum_{m \in \Zb} a_m x^{-m-1} \in (\End \VOA(\h))[\![x^{\pm1}]\!] 
\ee
on $\VOA(\h)$. These generate a VOSA such that $Y(a,x) = a(x)$ via the reconstruction theorem. We will denote this VOSA also by $\VOA(\h)$ -- this is the {\em symplectic fermion VOSA}.

Let $\{ \alpha^i \,|\,i=1,\dots,d\}$ be a basis of $\h$ and let $\{\beta^i\}$ be the dual basis with respect to $(-,-)$, such that $(\alpha^i,\beta^j)=\delta_{ij}$.
The stress tensor (or Virasoro element) is
\be\label{eq:SF-stress-tensor}
	T^{SF} = \tfrac12 \sum_{i=1}^d \beta^i_{-1} \alpha^i_{-1} \one \ ,
\ee
see \cite[Sect.\,1.9]{Frenkel:1987} and \cite[Sect.\,3.5]{Kac:1998}.
The corresponding Virasoro algebra has central charge $-d$.

As already remarked in \eqref{eq:V(d)=V0+V1}, $\VOA(\h)$ decomposes into a parity even and odd part as $\VOA(\h) = \VOA(\h)_\mathrm{ev} \oplus \VOA(\h)_\mathrm{odd}$. By construction, $\VOA(\h)_\mathrm{ev}$ is a (non-super) VOA, and it is shown in \cite{Abe:2005} that $\VOA(\h)_\mathrm{odd}$ is an irreducible $\VOA(\h)_\mathrm{ev}$-module.

\subsection{Realisation inside a lattice vertex operator super algebra}\label{sec:Vd-inside-lattice}

The even part of a single pair of symplectic fermions, the $W_2$-algebra, has a free field construction in terms of a free boson with modified stress tensor \cite{Fuchs:2003yu,Adamovic:2007er}. In more detail, consider $\Rb$ with its standard inner product and the lattice $2 \Zb \subset \Rb$. The $W_2$-algebra is defined as the kernel of a certain screening charge and is a sub-VOA of $\VOA_{2 \Zb}$. Here $\VOA_{2 \Zb}$ is the lattice-VOA for the lattice $2\Zb$, except that the Virasoro element is given by
\be
	T^\mathrm{FF} = \tfrac12\big( H_{-1}H_{-1} -H_{-2} \big) \one 
\ee
(`FF' for Feigin-Fuchs), where $H_m$ are the standard Heisenberg generators, $[H_m,H_n] = m\, \delta_{m+n,0} \,K$ and $K$ acts as $1$. The central charge of the resulting Virasoro representation is $c^{FF} = -2$.

Let $\h = \Cb^{0|2r}$, let $\{ \alpha^i | i=1,\dots,2r \}$ be the standard basis of $\Cb^{0|2r}$ and let the super-symmetric pairing on $\h$ be $(\alpha^i,\alpha^j) = \delta_{i+\frac d2,j} - \delta_{i,j+\frac d2}$. We write $\VOA(2r)$ instead of $\VOA(\h)$. 
A single pair of symplectic fermions, $\VOA(2)$, forms a sub-VOSA of the lattice VOSA $\VOA_{\Zb}$ given by $\Zb \subset \Rb$. Here, even and odd integers correspond to even and odd degree super-vector spaces. That is, as a representation of the Heisenberg algebra in $\svect$, $\VOA_{\Zb}$ is given by
\be
	\VOA_\Zb = \bigoplus_{m \in \Zb} \Pi^m F_m ~\in~ \svect \ ,
\ee
where $F_m$ is the highest weight module of the Heisenberg algebra with highest weight vector $v_m$ satisfying $H_0 v_m = m v_m$, and $\Pi^m$ is the identity for $m$ even and parity shift for $m$ odd. Note that $(\VOA_{\Zb})_\mathrm{ev} = \VOA_{2\Zb}$.

Let $Y(-,x) : \VOA_{\Zb} \to \mathrm{End}( \VOA_{\Zb}) [\![x^{\pm 1} ]\!]$ be the vertex operator of $\VOA_{\Zb}$. 
Following \cite{Kausch:1995py,Fuchs:2003yu,Adamovic:2007er}, we consider the following elements and their mode expansions: 
\be
	f^*  := v_1
	~ ,~~ f := -H_{-1} v_{-1} \quad , \quad 
	Y(f^*,x) = \sum_{m \in \Zb} f^*_m x^{-m-1} ~,~~
	Y(f,x) = \sum_{m \in \Zb} f_m x^{-m-1} \ .
\ee
One computes
\begin{align}
	Y(f^*,x)f &= - x^{-1} \exp\big(\sum_{m<0} \tfrac{x^{-m}}{-m} H_m \big)\exp\big(\sum_{m>0} \tfrac{x^{-m}}{-m} H_m \big) H_{-1}v_{-1} 
	\nonumber\\
	&= x^{-2} \one 
	- \tfrac12 (H_{-1} H_{-1} - H_{-2}) \one 
	+ 
	O(x) \ ,
	\label{eq:f*f-OPE}
\end{align}
while $Y(f,x)f$ and $Y(f^*,x)f^*$ are regular. For the modes of $f$ and $f^*$ we hence obtain the non-trivial (super-)bracket (i.e.\ an anticommutator, as $f,f^*$ are both parity-odd):
\be
	[f^*_m,f_n] = m \, \delta_{m+n,0} \, \id \ .
	\label{eq:sf-modes-act-VZ}
\ee
Thus $\VOA_\Zb$ becomes a $\widehat\h$-module (with $r=1$) where  $\alpha^1_m$ acts as $f^*_m$ and $\alpha^2_m$ as $f_m$.

On the vacuum $\one \in \VOA_{\Zb}$, we have $f_m \one = 0 = f^*_m \one$ for all $m \in \Zb_{\ge 0}$, so that $\one$ generates a highest weight module of $\widehat\h$ in $\VOA_{\Zb}$. Since  highest weight modules of $\widehat\h$ are simple 
(see, e.g., Section \ref{sec:brmon-on-rep}), we obtain an injective homomorphism of $\widehat\h$-modules $\iota : \VOA(2) \to \VOA_{\Zb}$. Since $Y(\alpha^i,x)$, $i=1,2$, generate $\VOA(2)$, this injective module homomorphism becomes an embedding of vertex super-algebras. To turn this into an embedding of VOSA, we need in addition:

\begin{lemma}
The Virasoro elements satisfy $\iota(T^{SF}) = T^{FF}$.
\end{lemma}

\begin{proof}
By \eqref{eq:SF-stress-tensor}, the stress tensor for one pair of symplectic fermions is $T^{SF} = - \alpha_{-1}^1 \alpha^2$. Thus $\iota(T^{SF}) = - f^*_{-1} f$. From \eqref{eq:f*f-OPE} we read off $-f^*_{-1} f = \tfrac12 (H_{-1} H_{-1} - H_{-2})\one$, as required.
\end{proof}

We can now take the $r$-fold cartesian product of the above construction. That is, we consider the lattice $\Zb^r$ inside $\Rb^r$ with its standard inner product. An element $(x_1,\dots,x_r) \in \Zb^r$ corresponds to a parity even (resp.\ odd) subspace of the lattice VOSA $\VOA_{\Zb^r}$ if $x_1+\cdots+x_r$ is even (resp.\ odd). We denote the generators of the $r$ copies of the Heisenberg algebra by $H^i_m$, $i=1,\dots,r$, $m \in \Zb$. To each $x \in \Zb^r$ there corresponds a highest weight module $F_x$ of the $r$ copies of the Heisenberg algebra, and we denote the highest weight vector by $v_x$. 
Let $e_1,\dots,e_r \in \Zb^r$ be the standard set of generators. With these notations in place, the above arguments establish the following theorem, which is the main result of this section.

\begin{theorem}\label{thm:symp-ferm-in-lattice}
There is a unique embedding of VOSAs 
$$\iota : \VOA(2r) \to \VOA_{\Zb^r} \ ,$$ 
such that $\iota(\alpha^i) = v_{e_i}$ and $\iota(\alpha^{r+i}) =  -H^i_{-1} v_{-e_i}$, $i=1,\dots,r$. The image of the Virasoro element is 
$\iota(T^{SF}) = \frac12 \sum_{i=1}^r(H_{-1}^iH_{-1}^i - H_{-2}^i)$.
\end{theorem}

The even subspace of $\VOA_{\Zb^r}$ is $\VOA_{D_r}$, where the root lattice $D_r$ was given in \eqref{eq:D-root-lattice}.
We immediately obtain:

\begin{corollary}\label{cor:symp-ferm-in-Dr-lattice}
The embedding of Theorem \ref{thm:symp-ferm-in-lattice} restricts to an embedding $\iota : \VOA(2r)_{\mathrm{ev}} \to \VOA_{D_r}$ of VOAs. 
\end{corollary}

We now turn to the relation to holomorphic VOAs. Recall that a VOA is {\em rational} if all its admissible modules are completely reducible,
and a rational VOA $\VOA$ is called {\em holomorphic} if every simple admissible $\VOA$-module is isomorphic to $\VOA$
(see e.g.\ \cite{Dong:2002}).
Lattice VOAs for even self-dual lattices are holomorphic \cite{Dong:1993}.
Via the embedding $\iota$ in Corollary \ref{cor:symp-ferm-in-Dr-lattice}, the lattice VOA $\VOA_{D_r^+}$ for the even self-dual lattice $D_r^+$ ($r \in 8 \Zb_{>0}$) contains $\VOA(2r)_{\mathrm{ev}}$ as a sub-VOA. Thus:

\begin{corollary}\label{cor:symp-even-in-lattice}
For $r \in 8 \Zb_{>0}$, the even part of the symplectic fermion VOA has a holomorphic extension, namely the lattice VOA of the even self-dual lattice $D_r^+$ with 
	Virasoro element
 as in Theorem \ref{thm:symp-ferm-in-lattice}.
\end{corollary}

This corollary explains the relation between projectively modular invariant combinations of characters of $\VOA(d)_{\mathrm{ev}}$ and the lattice $D^+_{d/2}$ made in Section \ref{sec:char-rel-so(d)1}. 

\section{Braided monoidal structure on representations}\label{sec:brmon-on-rep}

For the convenience of the reader, in this section we restate the relevant results from \cite{Davydov:2012xg,Runkel:2012cf}, see also the summaries in \cite{Davydov:2013ty,Gainutdinov:2015lja}.
Let $\h$ and $(-,-)$ be as in the previous section.

In addition to $\widehat\h$ we will need the ``twisted mode algebra'' $\widehat\h_\mathrm{tw}$. It is defined as $\widehat\h_\mathrm{tw} =  \h \otimes_\Cb (t^{\frac12} \Cb[t^{\pm1}]) \oplus \Cb K$. We again write $a_m = a \otimes t^m$, where now $m \in \Zb + \frac12$. The Lie bracket is as in \eqref{eq:SF-liebracket}, but now for $m,n \in \Zb + \frac12$. Let
\be
	\Rep^\mathrm{fd}_{\flat,1}(\widehat\h)
	\quad \text{and} \quad
	\Rep^\mathrm{fd}_{\flat,1}(\widehat\h_\mathrm{tw})
\ee
be the categories of all $\widehat\h$-modules and $\widehat\h_\mathrm{tw}$-modules in super-vector spaces, respectively, which in addition satisfy:
\begin{itemize}
\item[-] $K$ acts as 1.
\item[-] The modules are bounded below. (An $\widehat\h_\mathrm{(tw)}$-module $M$ is {\em bounded below} if for each $u \in M$ there is an $N>0$ such that $a^1_{m_1} \cdots a^k_{m_k}u=0$ for all $a^1,\dots,a^k  \in \h$ whenever $m_1+\cdots+m_k > N$.)
\item[-] The modules have finite-dimensional spaces of ground states. (The {\em space of ground states} of a module $M$ is the subspace on which all $a_n$ with $a \in \h$ and $n>0$ act as zero.)
\end{itemize}
Because they are bounded below, the modules in $\Rep^\mathrm{fd}_{\flat,1}(\widehat\h_\mathrm{(tw)})$ carry an action of the \hbox{Virasoro} algebra. 
Let $\{ \alpha^i \,|\,i=1,\dots,d\}$ be a basis of $\h$ and let $\{\beta^i\}$ be the dual basis with respect to $(-,-)$, such that $(\alpha^i,\beta^j)=\delta_{ij}$.
The Virasoro action resulting from the stress tensor \eqref{eq:SF-stress-tensor} is, for $m \in \Zb$, $m \neq 0$,
\be
  L_m = \tfrac12 \sum_{k \in \Zb-\delta} \sum_{j=1}^d \beta^j_k \alpha_{m-k}^j 
~~,\quad
L_0 = 
\sum_{k \in \Zb_{>0} - \delta} \sum_{j=1}^d \beta_{-k}^j \alpha_k^j
+ \begin{cases} 
\tfrac12\sum_{j=1}^d\beta^j_0 \alpha_{0}^j  & ;~\delta=0 \\
-\tfrac{d}{16} 
& ;~\delta=\tfrac12 \quad ,
\end{cases} 
\label{eq:virasoro-action}
\ee
where $\delta=0$ in the untwisted sector and $\delta=\tfrac12$ in the twisted sector (see e.g.\ \cite[Sect.\,1.9]{Frenkel:1987} for the purely even case). Note that for $m \neq 0$ there is no need for normal ordering in $L_m$, since the modes $\beta^j_k$ and $\alpha_{m-k}^j$ anti-commute.

\medskip

Given a $\h$-representation $A \in  \Rep^\mathrm{fd}(\h)$ and a super-vector space $B \in \svect^\mathrm{fd}$, we denote the corresponding induced $\widehat\h$- and $\widehat\h_\mathrm{tw}$-representations as
\be\label{eq:induced-mod-def}
	\widehat A := U(\widehat\h) \otimes_{\widehat\h_{\ge 0} \oplus \Cb K} A
	\quad , \quad
	\widehat B := U(\widehat\h_\mathrm{tw}) \otimes_{\widehat\h_{\mathrm{tw},>0} \oplus \Cb K} B \ ,
\ee
where $K$ is taken to act as $1$ on $A$ and $B$. 
For example, the induced module of the trivial $\h$-module $\Cb^{1|0}$ is just the symplectic fermion VOSA from \eqref{eq:free-superbos-VOSA} 
$\widehat{\Cb^{1|0}} = \VOA(\h)$.
It is shown in \cite[Thm.\,2.4\,\&\,2.8]{Runkel:2012cf} (following \cite[Sect.\,1.7]{Frenkel:1987} and \cite[Sect.\,3.5]{Kac:1998}) that induction provides equivalences of $\Cb$-linear categories
\be\label{eq:SF-iso-to-Rep}
	\Rep^\mathrm{fd}(\h) \xrightarrow{~\sim~}
	\Rep^\mathrm{fd}_{\flat,1}(\widehat\h)
	\quad \text{and} \quad
	\svect^\mathrm{fd} \xrightarrow{~\sim~}
	\Rep^\mathrm{fd}_{\flat,1}(\widehat\h_\mathrm{tw}) \ .
\ee
Here, $\Rep^\mathrm{fd}(\h)$ is the category of finite-dimensional representations of $\h$ -- understood as an abelian Lie super-algebra -- in super-vector spaces, and $\svect^\mathrm{fd}$ is the category of finite-dimensional super-vector spaces. 
We will abbreviate
\be
	\SF_0(\h) = \Rep^\mathrm{fd}(\h) 
	~~,\quad
	\SF_1(\h) = \svect^\mathrm{fd}
	~~,\quad
	\SF(\h) = \SF_0(\h) \oplus \SF_1(\h) \ .
\ee
For an object $C \in \SF(\h)$ which decomposes as $C = C_0 \oplus C_1$, $C_i \in \SF(\h)_i$, $i=0,1$, we set  $\widehat C :=  \widehat C_0\oplus \widehat C_1$.

The components $\SF_0(\h)$ and $\SF_1(\h)$ have two simple objects each, which we will denote as
\be\label{eq:SF-simples}
	\one , \Pi\one \in \SF_0(\h)
	\quad , \quad
	T , \Pi T \in \SF_1(\h) \ .
\ee
Here $\one = \Cb^{1|0}$ with trivial $\h$-action, $\Pi$ is parity exchange so that $\Pi\one = \Cb^{0|1}$. For $\SF_1(\h) = \svect^\mathrm{fd}$ we take the same super-vector spaces: $T = \Cb^{1|0}$, $\Pi T = \Cb^{0|1}$.

\medskip

Let $S(\h)$ be the symmetric algebra of $\h$ in $\svect$. 
Considered as a vector space, $S(\h)$ is the exterior algebra of $\h$
and
$\dim S(\h) = 2^d$. Since $\h$ is an abelian Lie super-algebra, $S(\h)$ coincides with the universal enveloping algebra $U(\h)$ and is a commutative and cocommutative Hopf algebra in $\svect$. Moreover
\be
	\Rep^\mathrm{fd}(\h) = \Rep^\mathrm{fd}(S(\h)) \ .
\ee
Denote its structure maps of $S(\h)$ by $\mu_S$ (multiplication), $\Delta_S$ (coproduct), $\eps_S$ (counit), and $S_S$ (antipode). For example, for $a \in \h$ the coproduct and antipode are $\Delta_S(a) = a \otimes \one + \one \otimes a$ and $S_S(a) = -a$.

It is shown in \cite[Sect.\,4.5\,\&\,5.2]{Davydov:2012xg} that $\SF(\h)$ is a ribbon category. We will recall here the tensor product, associator, braiding and ribbon twist, but we omit the rigid structure as its explicit form is not needed. 

In order to compare to \cite{Davydov:2012xg} and \cite{Runkel:2012cf} note that the pairing used in \cite{Runkel:2012cf} is called $(-,-)$ here and is natural from the mode algebra point of view, see \eqref{eq:SF-liebracket}. The pairing used in \cite{Davydov:2012xg} will be denoted by $(-,-)_{\SF}$ and is chosen to absorb the factors of $\pi i$ in the associator and braiding of $\SF(\h)$ appearing in \cite{Runkel:2012cf}. The two pairings are related by
\be\label{eq:two-pairings}
	(-,-) ~=~ \pi i \, (-,-)_{\SF} \ .
\ee
Denote by $C$ the copairing dual to $(-,-)_{\SF}$. That is, if $\{a_i \,|\, i = 1,\dots,d\}$ is a basis of $\h$ and $\{b_i \,|\, i = 1,\dots,d\}$ is the dual basis in the sense that $(a_i,b_j)_{\SF} = \delta_{i,j}$, then
\be\label{eq:def-copairing}
	C = \sum_{i=1}^{d} b_i \otimes a_i \ .
\ee
We will abbreviate $\hat C = \mu_S(C) = \sum_{i=1}^{d} b_i a_i$ and set
\be
	\gamma = \exp(C) \,\in\, S(\h)\otimes S(\h)
	\quad , \quad
	\sigma = \exp(\tfrac12 \hat C) \,\in\, S(\h) \ .
\ee
Recall that $S(\h)$ is also $\Zb$-graded (refining the $\Zb_2$-grading of $\svect$), with non-zero graded components in degrees $0,1,\dots,d$. The graded components in degree $0$ and $d$ are one-dimensional. We will make use of a particular cointegral $\lambda$ on $S(\h)$, which is non-zero only in the top degree (as are all cointegrals on $S(\h)$) and is normalised such that 
\be\label{eq:lambda-def}
	\lambda\big(\hat{C}^{\frac d2}\big) ~=~ 2^{\frac d2} \, \big(\tfrac{d}{2}\big)!  \ .
\ee

We write $* : \SF(h) \times \SF(h) \to \SF(h)$ for the tensor product functor and will reserve the notation $\otimes$ for the tensor product in $\svect$ and for that of $S(\h)$-modules in $\svect$. The tensor product $*$ is $\Zb_2$-graded, and depending on which sector $X,Y \in \SF(\h)$ are chosen from, it is defined to be:
\begin{equation}\label{eq:*-tensor}
X \ast Y ~=~
\left\{\rule{0pt}{2.8em}\right.
\hspace{-.5em}\raisebox{.7em}{
\begin{tabular}{ccll}
   $X$ & $Y$ & $X \ast Y$ &
\\
$0$ & $0$ & $X \otimes Y$ & $\in~\SF_0(\h)$
\\
 $0$ & $1$ & $X \otimes Y$ & $\in~\SF_1(\h)$
\\
 $1$ & $0$ & $X \otimes Y$ & $\in~\SF_1(\h)$
\\
$1$ & $1$ & $S(\h) \otimes X \otimes Y$ & $\in~\SF_0(\h)$
\end{tabular}}
\end{equation}
Here, sector by sector, the meaning of $\otimes$ is as follows. In sector $00$, it is the tensor product in $\Rep^\mathrm{fd}(S(\h))$, that is, $S(\h)$ acts on $X \otimes Y$ via the coproduct. 
In sectors $01$ and $10$, one forgets the $S(\h)$ action and the tensor product is just that of super-vector spaces. In sector 11, both tensor products are of  super-vector spaces, and the $S(\h)$ action is given by the left regular action on $S(\h)$.

\medskip

To describe the associator, braiding and ribbon twist, we need to 
fix a $\beta \in \Cb$ such that
\be\label{eq:beta4}
	\beta^4 = (-1)^{d/2} \ .
\ee
For symplectic fermions one has to choose $\beta = e^{-\pi i \, d/8}$, but the ribbon structure can be defined for all four possibilities of $\beta$, and we will keep it arbitrary (but fixed) unless we indicate otherwise. 
We remark that the four braided monoidal structures on $\SF(\h)$ (depending on $\beta$) we will now describe form an orbit under the action of the third abelian group cohomology $H^3_\text{ab}(\Zb_2,U(1)) \cong \Zb_4$, as is expected for a $\Zb_2$-graded tensor product functor.

The associator, braiding and twist
\be
	\alpha_{X,Y,Z} : X*(Y*Z) \to (X*Y)*Z
	~~,~~
	c_{X,Y} : X*Y \to Y*X
	~~,~~
	\theta_X : X \to X
\ee
depend on the sectors $X,Y,Z$ are taken from. The associator 
is trivial (i.e.\ that of $\svect$) in sectors 000, 001 and 100. The non-trivial components are, in string-diagram notation (read from bottom to top, see \cite[Sect.\,2]{Davydov:2012xg} for our conventions)
\be\label{eq:SF-assoc}
\begin{tabular}{c@{\hspace{0.9em}}|@{\hspace{0.9em}}c@{\hspace{0.9em}}|@{\hspace{0.9em}}c@{\hspace{0.9em}}|@{\hspace{0.9em}}c@{\hspace{0.9em}}|@{\hspace{0.9em}}c}
\small 010 & \small 011 & \small 101 & \small 110 & \small 111 \\
\hline
&&&& \\[-.5em]
\ipic{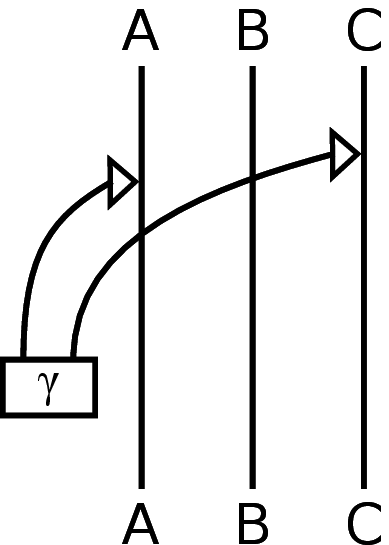}~ &
~\ipic{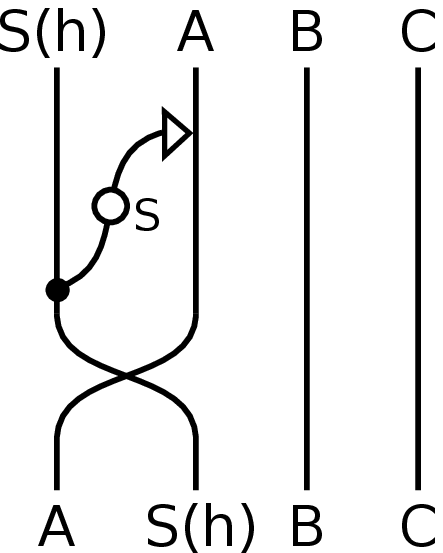}~ &
~\ipic{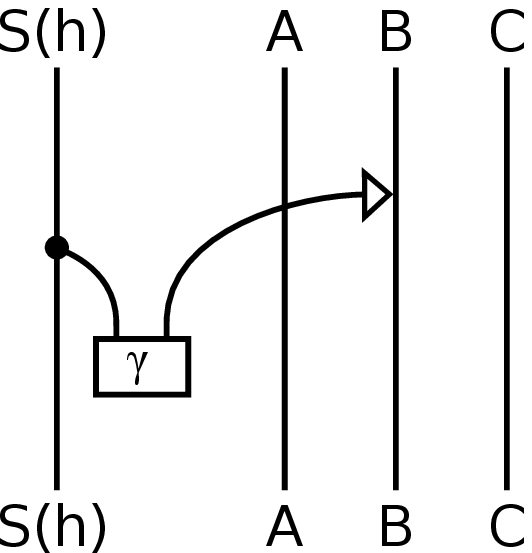}~ &
~\ipic{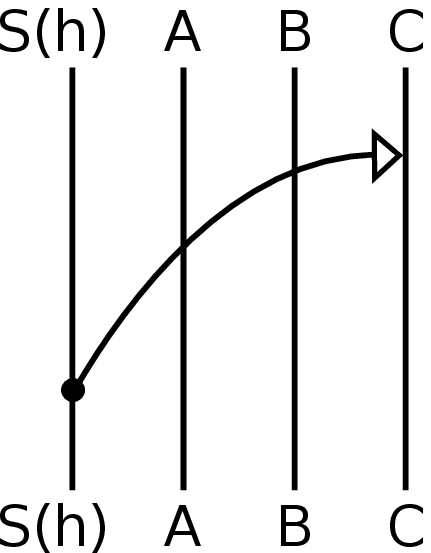}~ &
~\ipic{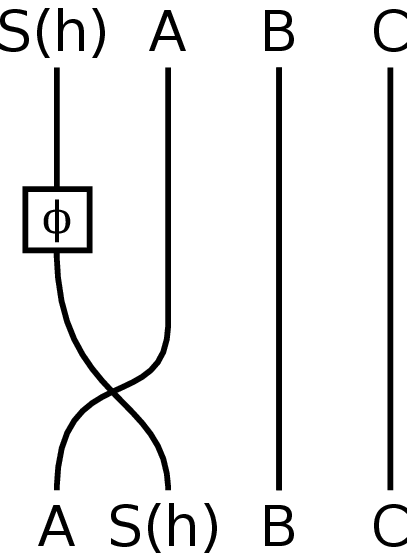} 
\end{tabular}
\ee
Here, $S := S_S$ is the antipode of $S(\h)$, $\phi : S(\h) \to S(\h)$ is given by 
\be
\phi = \beta^2 \, (\id \otimes (\lambda \circ \mu_S)) \circ (\gamma^{-1} \otimes \id)\ .
\ee 
The map $\phi$ is an isomorphism, and it maps the homogeneous subspace of $\Zb$-degree $k$ to that of degree $d-k$.

For a super-vector space $X$, we write $\omega_X : X\to X$ for its parity-involution. 
The braiding is given by
\be\label{eq:SF-braid}
\begin{tabular}{c@{\hspace{1.4em}}|@{\hspace{1.4em}}c@{\hspace{1.4em}}|@{\hspace{1.4em}}c@{\hspace{1.4em}}|@{\hspace{1.4em}}c}
\small 00 & \small 01 & \small 10 & \small 11  \\
\hline
&&& \\[-.5em]
\ipic{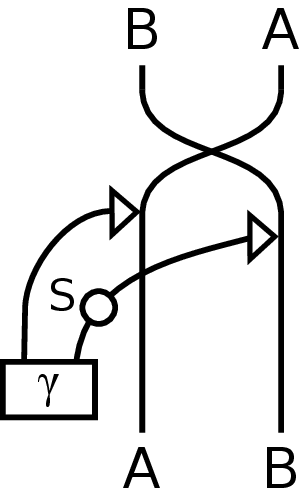} &
\ipic{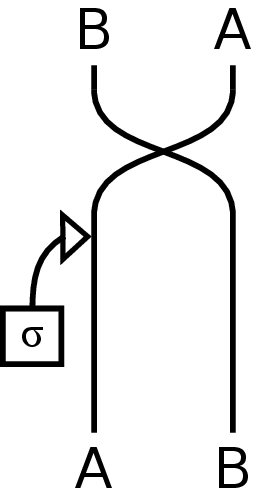} &
\ipic{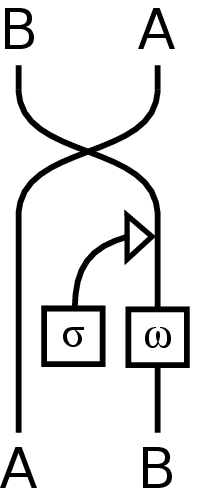} &
$\beta ~\cdot~$ \ipic{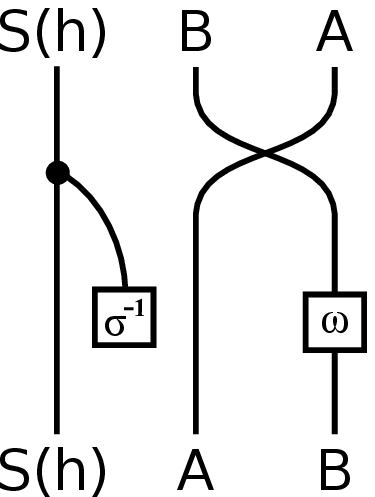} 
\end{tabular}
\ee
and finally the ribbon twist is
\be\label{eq:SF-twist}
\raisebox{.7em}{
\begin{tabular}{ccl}
 $X$ & $\theta_X : X \to X$
\\[.3em]
 $0$ & $\sigma^{-2}.(-)$
\\[.3em]
 $1$ & $\beta^{-1} \cdot \omega_X$
\end{tabular}}
\ee
On the four simple objects of $\SF(\h)$ the ribbon twists thus are
\be
	\theta_\one = \id
	~~,~~
	\theta_{\Pi\one} = \id
	~~,~~
	\theta_T = \beta^{-1} \id
	~~,~~
	\theta_{\Pi T} = - \beta^{-1} \id \ .
\ee
To relate this to the conformal weights given in \eqref{eq:V0-irrep-weights}, we note that our convention for the twist isomorphism on modules of a VOA is $\theta_U = e^{-2 \pi i L_0}.(-)$, and that we need to choose 
\be
	\beta = \beta_{SF} := e^{- \pi i d/8} \ .
\ee

Symplectic automorphisms $g \in Sp(\h)$ give rise to ribbon auto-equivalences $G_g$ of $\SF(\h)$ \cite[Rem.\,5.4]{Davydov:2012xg}. Namely, denote by $S(g)$ the automorphism of $S(\h)$ induced by $g$. On $\SF_0(\h)$, $G_g$ acts by pullback $S(g^{-1})^*$ of $S(\h)$-modules. On $\SF_1(\h)$, $G_g$ is the identity functor. The tensor structure on $G_g$ is $(G_g)_{A,B} : G_g(A * B) \to G_g(A)*G_g(B)$ with $(G_g)_{A,B} = \id$ unless both $A,B \in \SF_1(\h)$. For $A,B \in \SF_1(\h)$ we have
\be
	(G_g)_{A,B} : S(\h) \otimes A \otimes B \to  S(\h) \otimes A \otimes B
	\quad , \quad (G_g)_{A,B} = S(g) \otimes \id_A \otimes \id_B \ ,
\ee
see \cite[Sec.\,3.5]{Davydov:2012xg} for details. Note that acting with $S(g)$ is required to turn $(G_g)_{A,B}$ into an $S(\h)$-module morphism.

\section{Commutative algebras and Lagrangian subspaces}\label{sec:commalg}

Fix a $d$-dimensional purely odd complex super-vector space $\h$ with non-degenerate super-symmetric pairing $(-,-)$ as in Section \ref{sec:SF-VOSA}. 
In this and the next section we will work only in the ribbon category $\SF(\h)$
and will not make use of the symplectic fermion VOSA.

\medskip

An {\em algebra} (or {\em monoid}) in a monoidal category $\C$ is an object $A$ together with a morphism $\mu : A \otimes A \to A$ (the {\em multiplication}) and a morphism $\eta : \one \to A$ (the {\em unit}). They have to satisfy
\begin{align}
&\mu \circ (\id_A \otimes \mu) =
\mu \circ (\mu \otimes \id_A) \circ \alpha_{A,A,A} \ ,
&&\text{(associative)}
\nonumber
\\
&\mu \circ (\id_A \otimes \eta) = \rho_A \quad , \quad
\mu \circ (\eta \otimes \id_A) = \lambda_A \ ,
&&\text{(unital)}
\label{eq:unit+assoc-def}
\end{align}
where $\alpha,\lambda,\rho$ denote the associator and left and right unit constraint. Suppose that $\C$ is in addition braided with braiding $c$. Then $A$ is called {\em commutative} if $\mu \circ c_{A,A} = \mu$.

A {\em (left) $A$-module} is an $M \in \C$ together with a morphism $\rho : A \otimes M \to M$ satisfying associativity and unit conditions similar to \eqref{eq:unit+assoc-def}. The category of $A$-modules and $A$-module intertwiners will be denoted by ${_A\C}$. 
If $\C$ is braided and $A$ is commutative, we call an $A$-module $(M,\rho)$ {\em local} if $\rho \circ c_{M,A} \circ c_{A,M} = \rho$. The full subcategory of local modules will be called ${^\mathrm{loc}_A\C}$. 
	Local modules were introduced in \cite{Pareigis:1995} and called ``dyslectic'' there (see e.g.\ \cite{Frohlich:2003hm} for more details and references).

\medskip

The following is a wish list for the algebras $H$ in $\SF(\h)$ we are looking for:
\begin{enumerate}
\item[(W1)] 
In the decomposition $H = H_0 \oplus H_1$ with $H_i \in \SF_i(\h)$ we have $H_1 \neq \{0\}$.
\item[(W2)] $H$ has trivial twist, $\theta_H = \id_H$.
\item[(W3)] $H$ is commutative.
\item[(W4)] 
There exists $\kappa : H \to \one$ such that $\kappa \circ \mu : H * H \to \one$ is a non-degenerate pairing.
\item[(W5)] $H$ is simple as a left module over itself.
\end{enumerate}

W4 says that $H$ is a Frobenius algebra, and W2 and W3 imply that it is symmetric.

\begin{theorem}\label{thm:H-unique}
If $\beta=1$, then 
isomorphism classes of algebras $H$ satisfying W1--W5 are
	in bijection with
Lagrangian subspaces $\f\subset\h$, $H=H(\f)$.
The class in the Grothendieck ring of $\SF(\h)$ is 
	$[H(\f)] = 2^{\frac d2-1}\big\{[\one]+[\Pi\one]\big\} + [T]$, independent of $\f$.
	
\smallskip\noindent
If $\beta \neq 1$, conditions W1--W5 have no solution.
\end{theorem}

Note that  by \eqref{eq:beta4}, $\beta=1$ is possible only if $d \in 4\Zb$, 
and that W1 gets strengthened to $H_1 \cong T$.

The theorem will be proved in Sections \ref{sec:comm-twist}--\ref{sec:ex-unique}. We will assume that we are given an algebra $H$ satisfying W1--W5 and deduce its properties. As a result of this analysis, we will establish uniqueness and existence.

\subsection{Commutative multiplication and trivial twist}\label{sec:comm-twist}

We will write
\be 
  H = L \oplus E
  \qquad,\qquad  \text{where} \quad L \in \SF_0(\h) ~~,~~ E \in \SF_1(\h) \ .
\ee
By W1, $E$ is non-zero.

\begin{lemma}\label{lem:twist=id}
The following are equivalent:
\begin{enumerate}
\item $H$ has trivial twist,
\item $\hat C$ acts as zero on $L$ and either
\begin{enumerate}
\item $\beta = 1$ and $E$ is purely even, or
\item $\beta = -1$ and $E$ is purely odd.
\end{enumerate}
\end{enumerate}
\end{lemma}

\begin{proof}
The twist on $\SF(\h)$ was given in  \eqref{eq:SF-twist}. 
Since $E$ is non-zero by W1, and since $\theta_E = \beta^{-1} \,\omega_E$, the twist can only be trivial if $E$ is either purely even or purely odd. 
Since $\sigma^{-2} = \exp(-\hat C)$ and $\hat C$ is nilpotent, $\sigma^{-2}$ acts as the identity on $L$ if and only if $\hat C$ acts as zero.
\end{proof}

In simplifying expressions involving the parity of $E$, such as the braiding in $\svect$, we will use
	W2 and
the above lemma to replace
\be 
	\omega_E = \beta \, \id_E \qquad \text{where} \quad \beta = \pm 1 \ .
\ee

Next we turn to the multiplication map $\mu : H \ast H \to H$. It
has four components:
\be
\mu_{00} : L \ast L \to L
~~,\quad
\mu_{01} : L \ast E \to E
~~,\quad
\mu_{10} : E \ast L \to E
~~,\quad
\mu_{11} : E \ast E \to L
\ .
\ee 
We will express $\mu_{01}$ and $\mu_{10}$ in terms of even linear maps $\chi^{l,r} : L \to \End(E)$ via
\be\label{eq:mu01-10}
\mu_{01}(a \otimes e) = \chi^l(a)e
\quad , \quad
\mu_{10}(e \otimes a) = \chi^r(a)e
\qquad , \quad a \in L~~,~~ e \in E \ .
\ee
For $\mu_{11}$ note that $E*E = S(\h) \ot E \ot E$. Since $\mu_{11}$ is a map of $S(\h)$-modules, it is fixed uniquely by the image of $1 \in S(\h)$. Let $m : E \ot E \to L$ be given by $m(e \ot f) = \mu_{11}(1\ot e \ot f) \in L_\mathrm{ev}$. Then, for all $g \in S(\h)$,
\be\label{eq:mu_11-def-via-m}
	\mu_{11}(g \ot e \ot f) = g.m(e \ot f) \ .
\ee
Next we look at W3, $\mu \circ c_{H,H}=  \mu$. For sector 00, by \eqref{eq:SF-braid} this is equivalent to
\be\label{eq:mu00-commutative}
  \mu_{00}(a \otimes b) = \mu_{00} \circ \sigma^{\mathrm{s.v.}}_{L,L}(\gamma^{-1}.(a\otimes b))
  \qquad\qquad \text{for all}~~a,b \in L \ .
\ee
Here, $\sigma^{\mathrm{s.v.}}_{U,V} : U \otimes V \to V \otimes U$ denotes the symmetric braiding in $\svect$. Using Lemma \ref{lem:twist=id} one checks that for the other components of $\mu$, commutativity is equivalent to
\be\label{eq:mu-comm-01-10-11}
 \chi^l = \chi^r
 \qquad , \qquad m = m \circ \sigma^{\mathrm{s.v.}}_{E,E} \ .
\ee
We will write $\chi := \chi^l = \chi^r$. Note that $\chi$ is non-zero only on $L_\mathrm{ev}$. This will allow us to avoid a number of parity signs below.

\medskip

We denote the $\Zb$-degree $k$ subspace of $S(\h)$ by $S^k(\h)$, and write 
$S^{\ge n}(\h) = \bigoplus_{k=n}^d S^k(\h)$.
The following lemma will be important later:

\begin{lemma}\label{lem:C-trivial-only-above-d2-2}
Let $u \in S^k(\h)$
	be nonzero. 
If $\hat Cu=0$, then $k > \frac d2-2$.
\end{lemma}

\begin{proof}
Pick a basis $\{a_i\,|\,i = 1,\dots,d\}$ of $\h$ and consider the element $X = a_1 a_2 + a_3 a_4 + \cdots + a_{d-1} a_d \in S^2(\h)$.
Let  $k \le \frac d2-2$ and $u \in S^k(\h)$, $u \neq 0$,  be given. We will show that $Xu \neq 0$.

Fix a non-zero top form $t : S(\h) \to \Cb$ such that $t(a_1a_2\cdots a_d)=1$. Then \be\label{eq:top-form-pairing}
(g,h) = t(gh)
\ee
is a non-degenerate pairing on $S(\h)$.
Expand $u \in S^k(\h)$ with respect to the above basis as
\be
	u = \sum_{1 \le \delta_1 < \dots < \delta_k \le d} \xi_{\delta_1,\dots,\delta_k}  \,
	a_{\delta_1} \cdots a_{\delta_k} 
	\qquad , \quad \text{where}~~
	\xi_{\delta_1,\dots,\delta_k}  \in \Cb \ .
\ee
Pick a sequence $\delta_1< \dots < \delta_k$ for which $\xi_{\delta_1,\dots,\delta_k} \neq 0$
 and let $1\le \eps_1 < \eps_2 < \dots < \eps_{d-k} \le d$ be its complement in $\{1,2,\dots,d\}$. Then $(a_{\eps_1} \cdots a_{\eps_{d-k}} , u ) = \pm \xi_{\delta_1,\dots,\delta_k} \neq 0$.

Now consider all consecutive pairs $(2n{-}1,2n)$ in the sequence $Q = \{\eps_1,\dots,\eps_{d-k}\}$, i.e.\ let $P = \big\{n \,\big|\, \{2n{-}1,2n\} \subset Q \big\}$. 
The key point in the proof is the observation that since $k \le \frac d2-2$, $P$ must contain at least one element. 

Let $F = Q \setminus \bigcup_{n \in P} \{2n{-}1,2n\}$ be the set $Q$ minus all such pairs and
let $\nu_1 < \dots < \nu_{|F|}$ be the elements of $F$.
We now claim
\be
	a_{\eps_1} \cdots a_{\eps_{d-k}} \, u
	= \tfrac{1}{(|P|)!} \, a_{\nu_1} \cdots a_{\nu_{|F|}} X^{|P|} \, u
	\ .
\ee
To see this, write $Y = \sum_{n \in P} a_{2n{-}1} a_{2n}$ and note that 
	$a_{\nu_1} \cdots a_{\nu_{|F|}} X u = a_{\nu_1} \cdots a_{\nu_{|F|}} Y u$. 
The claim follows from 
\be
Y^{|P|} = (|P|)! \prod_{n \in P} a_{2n{-}1} a_{2n} \ .
\ee
Now we compute
\be
	\pm \xi_{\delta_1,\dots,\delta_k}
	=
	(a_{\eps_1} \cdots a_{\eps_{d-k}} , u )
	= 
	\tfrac{1}{(|P|)!} \, (a_{\nu_1} \cdots a_{\nu_{|F|}} X^{|P|-1},Xu)
\ee
Thus $Xu \neq 0$. 

Note that up to this point we have not made use of the symplectic form $(-,-)_{\SF}$ on $\h$ or of $\hat C$. Assume now that $\{a_i \,|\, i=1,\dots,d\}$ is a symplectic basis of $\h$ such that the dual basis $\{b_i\,|\, i=1,\dots,d\}$ is $b_1=a_2$, $b_2=-a_1$, etc. Then $\hat C = -2 X$, with $X$ as above.
\end{proof}

\subsection{Unit and associativity conditions}

Let $1_H = \eta(1) \in L$.
The unit condition in \eqref{eq:unit+assoc-def} is easily seen to be equivalent to, for all $a \in L$,
\be\label{eq:1L-unit}
	\mu_{00}(a \otimes 1_H) = a = \mu_{00}(1_H \otimes a)
	\quad , \quad 
	\chi(1_H) = 1 \ .
\ee

Analysed sector by sector, associativity in \eqref{eq:unit+assoc-def} gives eight conditions on $\mu_{00}$, $\chi$ and $m$. These follow by straightforward computation after substituting the associator in \eqref{eq:SF-assoc}. To state the result, we write, for $a,b \in L$, $\mu_{00}(a \otimes b) =: a \Lmul b$, and introduce the map $P : S(\h) \ot L \ot E \ot E \to L$, 
\be\lb{p}
	P(g \otimes a \ot e \ot f) = g.m\big(\,\chi(a)e\, \ot \,f\,\big) \ .
\ee
The $S(\h)$-action on $L$ will be denoted by $\rho^L : S(\h) \otimes L \to L$ or just by `$\,.\,$'.
Then the associativity conditions are:
\begin{align}
A000~:\quad & a \Lmul (b \Lmul c) = (a \Lmul b) \Lmul c & (\forall a,b,c \in L)
\nonumber\\[.4em]
A001~:\quad &\chi(a \Lmul b) = \chi(a)\chi(b)  & (\forall a,b \in L)
\nonumber\\[.4em]
A100~:\quad &\chi(a \Lmul b) = \chi(b)\chi(a)  & (\forall a,b \in L)
\nonumber\\[.4em]
A010~:\quad & \textstyle \chi(a)\chi(b) = \sum_{(\gamma)} (-1)^{|\gamma^{(2)}||a|} \chi(\gamma^{(2)}b)\chi(\gamma^{(1)}a) 
& (\forall a,b \in L)
\nonumber\\[.4em]
A110~~\quad &
P \circ \big\{ \id_{S(\h)} \otimes \big[ \rho^L \circ (S_S \otimes \id_L)\big] \ot \id_{E \ot E} \big\} \circ (\Delta_S \otimes \sigma^{\mathrm{s.v.}}_{E \ot E,L})
\nonumber\\
\&\,A011\,:&\qquad =
\mu_{00} \circ \big\{ \big[\rho^L \circ (\id_{S(\h)} \otimes m)\big] \otimes \id_L \big\}
\nonumber\\
&\qquad =
\mu_{00} \circ \sigma^{\mathrm{s.v.}}_{L,L} \circ \big\{ \big[\rho^L \circ (\id_{S(\h)} \otimes m)\big] \otimes \id_L \big\}
& \hspace{-4em}: ~ S(\h) \ot E \ot E \otimes L \to L
\nonumber\\[.4em]
A101~:\quad & P = P \circ \big\{ \gamma.(-) \otimes \sigma^{\mathrm{s.v.}}_{E,E} \big\}
& \hspace{-7em}: ~ S(\h) \ot L \ot E \ot E \to L
\nonumber\\[.4em]
A111~:\quad & \chi\big(g.m(u\ot v)\big)w = \beta \,\chi\big(\phi(g).m(u \ot w)\big)v
& \hspace{-2em}(\forall g \in S(\h), u,v,w \in E)
\nonumber
\end{align}
In deriving A110 we employed \cite[Lem.\,2.3\,(a)]{Davydov:2012xg} and \eqref{eq:mu-comm-01-10-11}. 
The first important consequence of the associativity conditions is:

\begin{lemma}
Let $J \subset L$ be the image of $\mu_{11} = \rho^L \circ (\id_{S(\h)} \otimes m) : S(\h) \ot E \ot E \to L$. Then $J \oplus E$ is a submodule of $H$, seen as a left-module over itself.
\end{lemma}

\begin{proof}
We need to check that the image of $\mu : (L \oplus E) \otimes (J \oplus E) \to (L \oplus E)$ is  contained in $J \oplus E$. We can again proceed sector by sector. For sector 11 this is true by definition and for sectors 01 and 10 there is nothing to check.

The only non-trivial check arises in sector 00, where we need to show $a \Lmul x \in J$ for all $a \in L$, $x \in J$. To do so, write $x = \sum_i g_i.m(e_i \ot f_i)$ for some $g_i \in S(\h)$, $e_i,f_i \in E$. Then $a \Lmul x$ is (possibly up to a parity sign) just the last line of A110 applied to $\sum_i g_i \ot e_i \ot f_i \ot a$. But the image of the first line of A110 is obviously in $J$.
\end{proof}

Thus W5 implies:

\begin{corollary} \label{cor:m-generates}
$\mu_{11} : S(\h) \ot E \ot E \to L$ is surjective. 
\end{corollary}

From this surjectivity result and the last two lines of A110 we can conclude that $L$ is commutative in $\svect$, i.e.\ that 
\be
\mu_{00} \circ \sigma^{\mathrm{s.v.}}_{L,L} = \mu_{00} \ .
\ee

\subsection{Non-degenerate pairing on $L$}

By W4 there is an $S(\h)$-module map $\kappa : L \to \one$ such that $\kappa \circ \mu_{00} : L \otimes L \to \one$ is non-degenerate.

\begin{lemma}\label{lem:kernel-hatC}
There is an even linear map $\tilde m : E \ot E \to S(\h)\ot (E \ot E)^*$ such that the map $L \to S(\h)\ot (E \ot E)^*$, $g.m(e \ot f) \to g.\tilde m(e \ot f)$ for all $g \in S(\h)$, $e,f \in E$, is well-defined and an injective $S(\h)$-module map.
\end{lemma}

\begin{proof}
Write $D : L \otimes L \to \one$, $D = \kappa \circ \mu_{00}$ for the non-degenerate pairing on $L$.

Let $\lambda$ be a non-zero element of  $S^d(\h)$. Then $\lambda$ is a cointegral for $S(\h)$ and $\one \to S(\h) \otimes S(\h)$, 
	$1 \mapsto \Delta_S(\lambda)$ 
is a non-degenerate copairing and a morphism in $\SF_0(\h)$.
Combine this with the copairing $\one \to (E \ot E)^* \ot (E \ot E)$ to get a non-degenerate copairing
\be
	B ~:~ \one ~\longrightarrow~ \big\{S(\h)\ot (E \ot E)^*\big\} \ot \big\{S(\h)\ot E \ot E \big\}
\ee
in $\SF_0(\h)$.

Since $\mu_{11} : S(\h) \ot E \ot E \to L$ is surjective by Corollary \ref{cor:m-generates}, we see that
\be
 (\id \otimes D) \circ (\id\otimes \mu_{11} \otimes \id) \circ (B \otimes \id)
	~:~ L \longrightarrow S(\h)\ot (E \ot E)^*
\ee
is an injective $S(\h)$-module map. Taking $\tilde m$ to be the composition of $m$ with the above map gives the statement of the lemma.
\end{proof}

Thanks to the above lemma we may as well take $m=\tilde m$, i.e.\ assume that $L = \langle m(E \ot E) \rangle \subset S(\h) \ot (E \ot E)^*$ as an $S(\h)$-module. Then, since $\eta : \one \to L$ is a map in $\SF_0(\h)$, and since only the top degree $S^d(\h)$ has trivial $\h$-action, we have
\be\label{eq:1L-in-Sd}
	1_H \,\in\, S^d(\h) \ot (E \ot E)^* \,\cap\, L \ .
\ee

Define the even linear map $\psi : S(\h)\ot E \ot E \rightarrow \End(E)$, $\psi(g \ot e \ot f)=\chi(g.m(e \ot f))$.

\begin{lemma}\label{lem:m-xi-chi}
\begin{enumerate}
\item $m(E \ot E) \subset S^{\ge\frac d2}(\h)\ot (E \ot E)^*$.
\item $\psi$ is nonzero on $S^{\frac d2}(\h)\ot E \ot E$ and zero on all other $S^{k}(\h)\ot E \ot E$.
\end{enumerate}
\end{lemma}

\begin{proof}
1.\ By Lemma \ref{lem:twist=id}, $\hat C$ acts trivially on $m(E \ot E)$. By Lemma \ref{lem:C-trivial-only-above-d2-2} and since $m$ is parity even, $m(E \ot E) \subset S^{\ge\frac d2}(\h)\ot (E \ot E)^*$. 

\smallskip\noindent
2.\ By part 1, any $g \in S^{>\frac d2}(\h)$ acts trivially on $m(E \ot E)$ and so $\psi$ is zero on $S^{>\frac d2}(\h)\ot E \ot E$. 

On the other hand, $\chi(1_H) = \id_E$ by \eqref{eq:1L-unit} and $1_H \in S^d(\h) \ot (E \ot E)^*$ by \eqref{eq:1L-in-Sd}. Since $1_H \in \langle m(E \ot E) \rangle$ there are $g_i \in S(\h)$, $e_i,f_i \in E$ such that $1_H = \sum_i g_i.m(e_i \ot f_i)$ and hence $\sum_i \psi(g_i \ot e_i \ot f_i)= \id_E$. It follows that $\psi$ is nonzero on at least one $S^k(\h)\ot E \ot E$ with $0 \le k \le \frac d2$.
Suppose $\psi$ is nonzero on $S^k(\h)\ot E \ot E$ with $0 \le k < \frac d2$. By condition A111 -- since $\phi$ maps $S^k(\h)$ isomorphically to $S^{d-k}(\h)$ -- $\psi$ would be non-zero also on $S^{d-k}(\h)\ot E \ot E$, in contradiction to $\psi$ being zero  on $S^{>\frac d2}(\h)$. 
\end{proof}

\begin{lemma}\label{lem:prop-to-1}
For all $u \in E \ot E$ and $g \in S^{\frac d2}(\h)$ we have $g.m(u) \in \Cb 1_H$.
\end{lemma}

\begin{proof}
By W5, $H$ is simple as a module over itself. By reciprocity, $\dim\SF(\one,H)=1$. Since the action of $S(\h)$ on $\one$ is trivial, the image of any map $\one \to L \subset S(\h) \ot (E \ot E)^*$ must lie in $S^d(\h) \ot (E \ot E)^*$. Thus $\SF(\one,H) \cong  L \cap \big[ S^d(\h) \ot (E \ot E)^* \big]$ is one-dimensional, and hence the latter is spanned by $1_H$. By part 1 of Lemma \ref{lem:m-xi-chi}, $g.m(u)$ is contained in $S^d(\h) \ot (E \ot E)^*$ (and of course in $L$), hence proportional to $1_H$.
\end{proof}

The proof of the next proposition requires the following technical lemma:

\begin{lemma}\label{lem:mult-of-nonzero-is-nonzero}
Let $\C$ be an abelian rigid monoidal category and let $A$ be a Frobenius algebra in $\C$ that is simple as a left module over itself. Then for all $U,V \in \C$ and all non-zero maps $u : U \to A$, $v : V \to A$, also the map $\mu \circ (u \otimes v) : U \otimes V \to A$ is non-zero.
\end{lemma}

\begin{proof}
Consider the $A$-module map  (we omit coherence isomorphisms)
\be
f = \big[ A \xrightarrow{\id \otimes \coev_V} A \ot V \ot V^* 
\xrightarrow{\id \ot v \ot \id} A \ot A \ot V^* \xrightarrow{\mu \ot \id} A \ot V^* \big] \ .
\ee
Suppose that $\mu \circ (u \otimes v) =0$. Then $f \circ u = 0$ and so $f$ has a non-zero kernel. Since $A$ is simple as an $A$-module, we must have $f=0$. But by non-degeneracy of the Frobenius pairing $\eps \circ \mu : A \otimes A \to \one$, the composite $(\eps \otimes \id) \circ f$ is non-zero, which is a contradiction.
\end{proof}

After these preparations, we can show a key restriction on $H$: 

\begin{proposition}\label{prop:dimE=1|0}
$E$ is purely even and one-dimensional.
\end{proposition}

\begin{proof}
As in the proof of part 2 of Lemma \ref{lem:m-xi-chi}, we can find $g_i \in S(\h)$, $e_i,f_i \in E$ such that $\sum_i \psi(g_i \ot e_i \ot f_i)= \id_E$.
In particular, there must be at least one $i_0$ such that $\psi(g_{i_0} \ot e_{i_0} \ot f_{i_0}) \neq 0$. We write $g = g_{i_0}$,  $e = e_{i_0}$ and  $f = f_{i_0}$. We may assume that $g = a_1 \dots a_{d/2}$ for appropriate $a_1,\dots,a_{d/2} \in \h$ (it is certainly a sum of elements of this form, hence at least one must give a non-zero contribution when applying $\psi$). Using Lemma \ref{lem:prop-to-1}, we see that we can modify $e$ or $f$ by a constant to achieve
\be
	a_1 \cdots a_{d/2} . m(e \ot f) = 1_H \ .
\ee
Let $u \in E \ot E$ be arbitrary. We will now evaluate A110 on $a_1 \cdots a_{d/2} \ot u \ot m(e \ot f)$.
To compute the left hand side, note that of all summands in $\Delta_S(a_1 \cdots a_{d/2})$, only $1 \otimes a_1 \cdots a_{d/2}$ gives a non-zero contribution, since by part 2 of Lemma \ref{lem:m-xi-chi}, $\chi(g.m(e \ot f))$ is zero unless $g \in S^{\frac d2}(\h)$. With this observation one easily checks the left hand side of A110 to be equal to $m(u)$. Altogether, A110 becomes
\be
	m(u)
	=
	\big\{ a_1 \cdots a_{d/2}.m(u) \big\} \Lmul m(e \ot f) \ .
\ee
By Lemma \ref{lem:prop-to-1}, there is a linear function $F : E \ot E \to \Cb$ such that $a_1 \cdots a_{d/2}.m(u) = F(u) 1_L$. Writing $m_0 := m(e \ot f)$, the above equality becomes $m(u) = F(u) \,  m_0$. In particular, the image of $m$ is one-dimensional. 

If $\dim E \ge 2$ we can find 
	non-zero
elements $p,q \in E$ such that $F(p \ot q)=0$. Hence also $m(p \ot q)=0$. 
	Thinking of $p,q$ as non-zero maps $\Cb^{1|1} \to E$ in $\SF_1(\h)$, it follows that $\mu_{11} \circ (p * q) = 0$ (as $\mu_{11}(g \ot p \ot q) = g.m(p\ot q)=0$ for all $g \in S(\h)$). Together with condition W5, this is a contradiction to Lemma \ref{lem:mult-of-nonzero-is-nonzero}. 

Hence $\dim E=1$. If $E$ were purely odd, the symmetry requirement \eqref{eq:mu-comm-01-10-11} on $m$ would force $m$ to be identically zero. But then also $\mu_{11}=0$, in contradiction to the non-degeneracy condition W4 (and also to Lemma \ref{lem:mult-of-nonzero-is-nonzero})
\end{proof}

By the above proposition, we can restrict our attention to the case 
\be
\beta = 1
\quad \text{and} \quad
E=T \ .
\ee 
Then $m : E \ot E \to L$ is just an 
	even
element $m \in L$, and Lemma \ref{lem:m-xi-chi} says that $L = \langle m \rangle \subset S^{\ge \frac d2}(\h)$.

\begin{lemma}\label{lem:m-homog}
$m \in S^{\frac d2}(\h)$, and $\chi$ is nonzero on $S^d(\h)$ and zero on all other $S^k(\h) \cap \langle m \rangle$.
\end{lemma}

\begin{proof}
Write $m$ as a  sum of its homogeneous components and let $m_M \in S^M(\h)$ be the non-zero component of maximal $\Zb$-degree. Since $m \in S^{\ge\frac d2}(\h)$ we have $M \ge \frac d2$. By non-degeneracy of the pairing \eqref{eq:top-form-pairing} there exists a $g \in S^{d-M}(\h)$ such that $gm_M = 1_H$, and therefore $\chi(gm_M) = 1$. By part 2 of Lemma \ref{lem:m-xi-chi} we must have $d-M = \frac d2$. Altogether we obtain $m \in S^{\frac d2}$.

That $\chi$ is non-zero only on $S^d(\h)$ now follows from Lemma \ref{lem:m-xi-chi}.
\end{proof}

Given a subspace $\mathfrak{u} \subset \h$ of dimension $k = \dim\mathfrak{u}$, $S^k(\mathfrak{u})$ is a one-dimensional subspace of $S^k(\h)$. In terms of a basis $\{u_1,\dots,u_k\}$ of $\mathfrak{u}$ it is given by
\be
	S^k(\mathfrak{u}) = \Cb \, u_1 u_2 \cdots u_k \ .
\ee
For an $S(\h)$-module $V$ the {\em annihilator of $V$ in $\h$} is defined to be
\be
	ann_\h(V) = \{ x \in \h \,|\,x.V = 0 \} \ .
\ee

\begin{lemma}\label{lem:kernel-is-Lag}
Abbreviate $\f := ann_\h(L)$. Then
\begin{enumerate}
\item $m \in S^{\frac d2}(\f)$~,
\item $\f$ is a Lagrangian subspace of $\h$ (and so $\dim\f = \frac d2$),
\item $L = \{ g\in S(\h) \,|\, g\f=0  \}$, that is, $L$ is the annihilator of $\f$ in $S(\h)$.
\end{enumerate}
\end{lemma}

\begin{proof}
Since $m$ generates $L$ as a $S(\h)$-module, we can equally write $\f = \{ x \in \h \,|\,xm = 0 \}$.

\smallskip\noindent
1.\ Let $n = \dim\f$ and let $\{a_1,\dots,a_n\}$ be a basis of $\f$. Extend this to a basis $\{a_1,\dots,a_d\}$ of $\h$ and let $\{b_1,\dots,b_d\}$ be the dual basis with respect to $(-,-)_{\SF}$. The copairing $C$ takes the form \eqref{eq:def-copairing}.

Evaluate condition A101 on $1 \otimes gm$ for $g \in S^{\frac d2-1}(\h)$.
By Lemma \ref{lem:m-xi-chi}, $P(1 \otimes gm)=0$ (here $P$ is as in \ref{p}). Furthermore, $P(\gamma.(1 \otimes gm)) = P(C.(1 \otimes gm))$, so that
\be
	\forall g \in S^{\frac d2-1}(\h) ~:\quad
	\sum_{i=1}^d \chi(a_igm) \, b_im = 0 \ .
\ee
In other words, for all $g \in S^{\frac d2-1}(\h)$ we have $\sum_{i=n+1}^d \chi(ga_im) \, b_i \in \f$, where we used $a_im=0$ for $i=1,\dots,n$ and that $a_i g = \pm g a_i$. By construction, the elements $a_im \in S^{\frac d2+1}(\h)$, $i=n+1,\dots,d$ are linearly independent. 
Using non-degeneracy of the pairing $(g,h) \mapsto \chi(gh)$ on $S(\h)$,
we see that there are $g_i \in S^{\frac d2-1}(\h)$, $i=n+1,\dots,d$, such that $\chi(g_ia_jm) = \delta_{ij}$. Hence we must have 
\be\label{eq:m-from-lag-aux2}
	b_{n+1},\dots,b_d \in \f \ .
\ee
We conclude that $n \ge \tfrac d2$. 

Expand $m$ with respect to the above basis,
\be\label{eq:m-from-lag-aux1}
	m = \sum_{\eps_1 , \dots , \eps_d \in \{0,1\} , \sum_{i=1}^d \eps_i = \frac d2} \xi_{\eps_1,\dots,\eps_d}  \,
	a_1^{\,\eps_1} \cdots a_d^{\,\eps_d}
	\qquad , \quad \text{where}~~
	\xi_{\eps_1,\dots,\eps_d}  \in \Cb \ .
\ee
Since $a_i m =0$ for $i=1,\dots,n$, we must have $\xi_{\eps_1,\dots,\eps_d} =0$ 
if any one of $\eps_1,\dots,\eps_n$ is zero.
Thus each non-zero summand in \eqref{eq:m-from-lag-aux1} must at least contain the product $a_1 \cdots a_n$. In particular, $n \le \frac d2$, which, when combined with the above estimate, gives $n  = \frac d2$ and
\be\label{eq:m-product-ai}
	m \in \Cb \, a_1 \cdots a_n \ .
\ee

\noindent
2.\ By \eqref{eq:m-from-lag-aux2}, the basis elements dual to $a_{n+1},\dots,a_d$ lie in $\f$,
	and since $n = \frac d2$ they
form a basis of $\f$. Since $(a_i,b_j)=0$ for all $1 \le i \le n < j \le d$, it follows that $\f$ is Lagrangian. 

\smallskip\noindent
3.\ The condition that $a_ig=0$ for $i=1,\dots,\frac d2$ is equivalent the condition that each monomial in $g$ contains the factor $a_1\cdots a_{d/2}$. Since $L$ is generated by $m$, together with \eqref{eq:m-product-ai} this shows the claim.
\end{proof}

\subsection{Uniqueness and existence}\label{sec:ex-unique}

\subsubsection*{Uniqueness}

\begin{lemma}\label{lem:unique-from-f}
Let $H,H'$ be two algebras satisfying W1--W5, and let $\f = ann_\h(H_0)$, $\f' = ann_\h(H'_0)$. Then $H$ and $H'$ are isomorphic as algebras if and only if $\f=\f'$. 
\end{lemma}

\begin{proof}
We will write $H = L \oplus T$ and $H' = L' \oplus T'$ with $L,L' \in \SF_0(\h)$.

The implication `$\Rightarrow$' is clear. 

For `$\Leftarrow$', 
as in Lemma \ref{lem:kernel-hatC} (together with Proposition \ref{prop:dimE=1|0}), we will think of $H$ as a submodule of $S(\h)$ generated by $m \in S^{\frac d2}(\f)$ (Lemma \ref{lem:kernel-is-Lag}). Let $1_H \in S^{d}(\h)$ be the unit of $H$. We will first show that the multiplication $\mu$ of $H$ is uniquely determined by $m$ and $1_H$.

The map $\chi \in S^d(\h)^*$ is fixed by \eqref{eq:1L-unit}: $\chi(1_H)=1$.
Let $M : S(\h) \to L$ be the map $M(g) = gm$. By Corollary \ref{cor:m-generates}, $M$ is surjective and thus condition A110 determines $\mu_{00}$ uniquely.
The map $\mu_{11}$ is fixed by \eqref{eq:mu_11-def-via-m} to be $\mu_{11}=M$.

Let now $m',1_{H'} \in S(\h)$ be the corresponding elements for $H'$. Since $\f=\mathfrak{f'}$, we have $m,m' \in S^{\frac d2}(\f)$, and so there is a $\lambda_1 \in \mathbb{C}^\times$ with $m' = \lambda_1 \, m$. Analogously, since $1_H,1_H' \in S^d(\h)$, there is a $\lambda_0 \in \mathbb{C}^\times$ such that $1_{H'} = \lambda_0 \, 1_H$.
Consider the map $f : H \to H'$ with components $f_0 : L \to L'$ and $f_1 : T \to T'$ given by
\be
	f_0 = \lambda_0 \, \id_L 
	\quad , \quad
	f_1 = \sqrt{\lambda_0/\lambda_1} \,\, \id_T \ .
\ee
It is straightforward to check that $f$ is an algebra isomorphism in $\SF(\h)$.
\end{proof}

\subsubsection*{Existence}

Let $\f \subset \h$ be a Lagrangian subspace, and let $m \in S^{\frac d2}(\f)$, $1_H \in S^d(\h)$ be non-zero elements. Set $L := \langle m \rangle$ and $H := L \oplus T$. We will show that there is an algebra structure on $H$ which satisfies W1--W5.

Fix a symplectic basis $\{ a_1,\dots,a_d \}$ of $\h$ such that $\{a_{d/2+1},\dots,a_{d}\}$ spans $\f$, $(a_i,a_j)_{\SF} = \pm 1$ for $j = i \pm \frac d2$ and zero else, and $m = a_{d/2+1} \cdots a_d$. Then the copairing \eqref{eq:def-copairing} takes the form
\be\label{eq:C-for-existence-proof}
	C = - \sum_{i=1}^{d/2} a_i \otimes a_{i+d/2} + \sum_{i=d/2+1}^{d} a_i \otimes a_{i-d/2} \ .
\ee
Define $M(g) = gm$ and $\psi(g) = \chi(M(g))$.

\medskip\noindent
{\em Unit and multiplication:}  $\eta : \one \to H$ is fixed by the choice of $1_H$ and, as above, $\chi \in S^d(\h)^*$ is fixed by $\chi(1_H)=1$. The components $\mu_{01}$, $\mu_{10}$ and $\mu_{11}$ of the multiplication are defined by \eqref{eq:mu01-10} and \eqref{eq:mu_11-def-via-m}. It remains to check that 
	a (necessarily unique)
solution $\mu_{00}$ to A110 actually exists. This is done in the next lemma, for which we
denote the left hand side of A110 by $f : S(\h) \otimes L \to L$.

\begin{lemma}
$\ker (M \otimes \id_L) \subset \ker f$.
\end{lemma}

\begin{proof}
The kernel of $M$ is spanned by all words in the $a_i$ which contain at least one factor $a_j$ with $j > \frac d2$. It is therefore enough to check that $f(a_jg \otimes u)=0$ for all $j > \frac d2$, $g \in S(\h)$ and $u \in L$. This in turn follows since $\Delta(a_jg) = (a_j \otimes \one + \one \otimes a_j) \Delta(g)$ and the summand $(a_j \otimes \one)\Delta(g)$ gives zero as $a_j$ multiplies $m$, while the summand $(\one \otimes a_j) \Delta(g)$ gives zero as as $a_j$ acts as zero on $L$.
\end{proof}

We thus obtain a unique map $\mu_{00} : L \otimes L \to L$ such that 
\be\label{eq:existence-mu00-def}
	\mu_{00}(M(g) \otimes M(h)) = \sum_{(g)} 
	g'm \, \psi\big(S(g'')h\big) \ ,
\ee	
where we used Sweedler-notation $\Delta(g) = \sum_{(g)} g' \otimes g''$. It is easy to see that $\mu_{00}$ is an $S(\h)$-module map. 

At this point the unit and multiplication morphisms of $H$ are fixed, and we can proceed to verify unitality, associativity and W1--W5.

\medskip\noindent
{\em W1:} Holds by construction.

\medskip\noindent
{\em W2,W3: Commutativity and trivial twist:}  
Since $a_i$ acts as zero on $L$ for $i>\frac d2$,
it is clear from  \eqref{eq:C-for-existence-proof} that $C$ acts as zero on $L \otimes L$ and that $\hat C$ acts as zero on $L$.  
Lemma \ref{lem:twist=id} shows that $H$ has trivial twist if and only if $\beta=1$. 
For commutativity, we saw in Section \ref{sec:comm-twist} that for $E=T$, the only non-trivial condition arises in sector 00 (see \eqref{eq:mu00-commutative}). Since $\gamma$ acts as identity on $L \otimes L$, for commutativity of $L$  we need to show:

\begin{lemma}\label{lem:mu-super-comm}
The multiplication of $H$ satisfies $\mu_{00} = \mu_{00} \circ \sigma^{\mathrm{s.v.}}_{L,L}$.
\end{lemma}

\begin{proof}
Substituting \eqref{eq:existence-mu00-def} and using commutativity and cocommutativity of $S(\h)$, we see that commutativity of $L$ is equivalent to
\be\label{eq:two-ways-mu00}
	\sum_{(g)} g'm \, \psi\big(S(g'')h\big) 
	= \sum_{(h)} \psi\big(gS(h')\big) \, h''m  \ .
\ee
Composing both sides with the invertible endomorphism of $S(\h) \otimes S(\h)$ given by $g \otimes h \mapsto \sum_{(h)} gh' \otimes h''$ gives the equivalent condition
\be
	\sum_{(g)} g'h'm \, \psi\big(S(g''h'')h'''\big) 
	= \psi(g) \, hm  \ .
\ee
The left hand side simplifies further to $\sum_{(g)} g'hm \, \psi(g'')$ (the antipode can be omitted since it is given by parity involution and $\psi$ is an even map to $\Cb^{1|0}$). We now claim that
\be\label{eq:psi-omit-coproduct}
	\sum_{(g)} g'm \, \psi(g'') = m \, \psi(g) \ ,
\ee
which implies the lemma. To see this, first note that in the basis of $S(\h)$ arising from our choice $\{a_1,\dots,a_d\}$ above, $\psi$ is non-zero only on $a_1\cdots a_{d/2}$. This means that the only terms contributing to $\sum_{(g)} g' \, \psi(g'')$ have $g'$ being either $1$ or containing generators $a_i$ with $i > \frac d2$. Since the latter annihilate $m$, the claim follows.
\end{proof}

\medskip\noindent
{\em Unitality and associativity:}  That $1_H$ is the unit for $H$ is immediate in sectors 01 and 10, because $\chi(1_H)=1$. For sector 00 the unit property follows from the observation that $1_H$ is a cointegral for $S(\h)$.

Next we check the associativity conditions A000--A111. Condition A110 and A011 hold by definition of $\mu_{00}$ and by Lemma \ref{lem:mu-super-comm}. Conditions A010 and A101 are a consequence of $C$ acting trivially on $L \otimes L$. The equivalent conditions A100 and A001 follow 
	from composing \eqref{eq:existence-mu00-def} with $\chi$ and using \eqref{eq:psi-omit-coproduct}. 
The remaining two conditions require a small calculation:

\medskip\noindent
A000: Recall the two equivalent ways to write the product $\mu_{00}$ given in \eqref{eq:two-ways-mu00}. When substituting the product into $(M(g) \Lmul M(h)) \Lmul M(k)$ and $M(g) \Lmul (M(h) \Lmul M(k))$, first use the variant of $\mu_{00}$ where the coproduct does not act on $M(g)$ or $M(k)$ and then the other variant.
A000 then reduces to coassociativity of $S(\h)$.

\medskip\noindent
A111: Recall that $\phi$ maps $S^k(\h)$ to $S^{d-k}(\h)$. 
Thus, unless $g \in S^{\frac d2}(\h)$, both $\chi(g.m)$ and $\chi(\phi(g).m)$ are zero. Taking $g$ to be a monomial in the generators $\{a_i\}$, we see that $g.m =0$ unless $g \in \Cb\, a_1 \cdots a_{d/2}$. 

For $g \in S^{\frac d2}(\h)$, the explicit expression for $\phi(g)$ is (see \cite[Sect.\,5.2]{Davydov:2012xg})
\be
	\phi(g) = \beta^2 \, (-1)^{\frac{d(d+2)}8} \hspace{-.5em} \sum_{i_1<\cdots<i_{d/2}} \hspace{-.5em} \lambda(a_{i_1} \cdots a_{i_{d/2}} g) \, b_{i_1} \cdots b_{i_{d/2}} \ .
\ee
Here $\lambda$ was given in \eqref{eq:lambda-def} and $\{b_i\}$ is the dual basis: $b_k = a_{k+\frac d2}$ for $k=1,\dots,\frac d2$ and  $b_k = -a_{k-\frac d2}$ for $k=\frac d2+1,\dots,d$.
This shows that $\phi$ maps monomials in $\{a^i\}$ to monomials, and that the monomial $a_1 \cdots a_{\frac d2}$ is an eigenvector. It remains to show that the eigenvalue is $1$.

Similar to the computation in \cite[Sect.\,5.2]{Davydov:2012xg}
	(see also the proof of Lemma \ref{lem:C-trivial-only-above-d2-2})
, one first verifies that
\be
	\hat C^{\frac d2} =  (-1)^{\frac{d(d+2)}8} \, 2^{\frac d2} \big(\tfrac d2\big)! \,  a_1 \cdots a_d \ .
\ee
This shows that $\lambda(a_1 \cdots a_d) =  (-1)^{\frac{d(d+2)}8}$, and so indeed $\phi(a_1\cdots a_{d/2}) = a_1\cdots a_{d/2}$ (we used that $\beta=1$, and that since $d \in 4\Zb$, $a_1\cdots a_{d/2}$ is an even element).

\medskip\noindent
{\em W4: Non-degenerate pairing on $H$:}  
Let $\kappa : L \to \one$ be non-zero  on $m$ and zero on $S^{>\frac d2}(\h)$. This is a map of $S(\h)$-modules. Consider the pairing
\be
d := \kappa \circ \mu ~:~ H * H \to \one \ .
\ee
By definition, $d$ is non-degenerate if there exists a copairing $b : \one \to H*H$ such that
\be\label{eq:H-pairing-nondeg-def}
	(d * \id_H) \circ \alpha_{H,H,H} \circ (\id_H * b) = \id_H =
	(\id_H * d) \circ \alpha_{H,H,H}^{-1} \circ (b*\id_H)
\ee
Let us write $d$ as $d = d_0+d_1$ with $d_0:L *L \to \one$ and $d_1 : T*T \to \one$ and analyse non-degeneracy separately for $d_0$ and $d_1$.

For $d_0$, the above condition is equivalent to $d_0$ being non-degenerate as a bilinear map $L \otimes L \to \Cb$. To verify this we work instead with the pairing $d_0' := d_0 \circ (M \otimes M) : S(\h) \otimes S(\h) \to \Cb$. Substituting the definition of $\mu_{00}$ gives
\be
d'_0 = \kappa(m) \, \psi \circ \mu_S \circ (S_S \otimes \id) \circ (M \otimes M) \ .
\ee
$M$ maps the subspace (not submodule) of $S(\h)$ with basis $a_{i_1} \cdots a_{i_m}$ for $1 \le i_1 < i_2 <\dots < i_m \le \frac d2$ isomorphically to $L$. Since $\psi$ is non-zero only on $a_1 a_2 \cdots a_{d/2}$, we get on the above basis:
\be
	d_0'(a_{i_1} \cdots a_{i_m},a_{j_1} \cdots a_{j_n}) =
	\begin{cases}
	\neq 0 & m+n=\tfrac d2 \text{ and } \{i_1,\dots,i_m,j_1,\dots,j_n \} = \{1,2,\dots,\frac d2\}
	\\
	0 & \text{else}
	\end{cases}
\ee
Thus the pairing $d_0$ is non-degenerate on $L \otimes L$.

For $d_1$, we need to work with \eqref{eq:H-pairing-nondeg-def}. For the copairing $b_1 : \one \to T*T = S(\h)$ we have, up to normalisation, only one choice: to be compatible with the $S(\h)$-action, $b_1$ has to map 1 to a top-degree element $t$ of $S(\h)$. The first equality in \eqref{eq:H-pairing-nondeg-def} becomes
\be\label{eq:pairing-nondeg-TT}
	\kappa(\phi(t)m) = 1 \ .
\ee
Since $\phi(t)$ is non-zero and in degree 0, $\phi(t)m$ is non-zero and proportional to $m$. We can now choose the normalisation of $t$ such that \eqref{eq:pairing-nondeg-TT} holds.
The second equality in  \eqref{eq:H-pairing-nondeg-def} follows from the first since $c_{T,T} \circ b_1 = b_1$ (as $\sigma$ acts as the identity on $t$ and $\beta=1$) and since $\mu$ is commutative.

\medskip\noindent
{\em W5: $H$ is simple as a left module over itself:} We will show that any morphism $f : H \to X$ of $H$-modules in $\SF(\h)$ is injective or zero. 

Suppose $f$ is zero on $T$. Since it is an $H$-module map, it must also be zero on the image of $\mu_{11}$ and in particular on $m$. But $m$ generates $L$ as an $S(\h)$-module, and so $f$ is zero on $L$ as well.

Suppose now that $f$ is non-zero on $T$. Then 
$f$ has to be non-zero on $1_H$, since $1_H$ generates $H$ as an $H$-module. 
Since $1_H \in L = \langle m \rangle$, $f$ is non-zero also on $m \in L$. But then $f$ has to be injective on $L$, for if a non-zero $u \in L$ is mapped to zero, then also $f(gu)=0$ for all $g \in S(\h)$, but there is a $g$ such that $gu=1_H$. Since $T$ is simple and $f$ is non-zero on $T$, it is also injective on $T$.

\medskip

We denote the algebra $H$ constructed above by $H(\f)$. It is determined by $\f$ up to isomorphism.
This completes the proof of Theorem \ref{thm:H-unique}.

\section{Local modules of $H(\f)$}\label{sec:locmod}

Recall from Section \ref{sec:commalg} the notation ${_A\C}$ (resp.\ ${^\mathrm{loc}_A\C}$) for modules (resp.\ local modules) over an algebra $A$ in a (braided)
	monoidal
category $\C$. 
Let $\beta = 1$ (and hence $d \in 4\Zb$) 
	and fix a Lagrangian subspace $\f\subset\h$. Write $H = H(\f)$ for the corresponding algebra as given by Theorem \ref{thm:H-unique}.

\subsection{Simple modules}

Let $k$ be an algebraically closed field.
A {\em tensor category} is a rigid monoidal $k$-linear abelian category with simple unit.
A tensor category is  {\em finite} if it  is equivalent as a $k$-linear category to the category of finite dimensional modules over a finite dimensional $k$-algebra \cite{eo}.
Internally the finiteness is finite length together with the existence of projective covers for all (simple) objects (see \cite{eo} for details).

For example, the category $\SF(\h)$ of symplectic fermions is a finite tensor category. Indeed, $\SF_0(\h)$ is the category of modules over the bosonisation of $S(\h)$, while $\SF_1(\h)$ is the category of modules over the group algebra of the group of order 2. Rigidity is shown in \cite[Sect.\,3.6]{Davydov:2012xg}.

\ble\lb{sm} 
Let $A$ be an algebra in a finite tensor category $\C$.
For any simple $A$-module $M$ there is a simple object $S\in\C$ together with an epimorphism of $A$-modules $A\ot S\to M$.
\ele
\bpf
Let $S\subset M$ be a simple subobject of $M\in\C$. The module map $A\ot M\to M$ induces a non-zero morphism $A\ot S\to M$, which is a morphism of $A$-modules. Its image is a 
	non-zero
submodule of $M$. Hence the image coincides with $M$.
\epf

\bre\lb{cfm}
It follows from Lemma \ref{sm} that any simple $A$-module is a direct summand of the cosocle (the maximal semi-simple quotient) of the free $A$-module
	$A\ot S$ 
for a simple object $S\in\C$.
\ere

Now let $\C=\SF(\h)$ be the category of symplectic fermions and let $A=H$
	as fixed above.
The four simple objects of $\SF(\h)$ give rise to the free $H$-modules
\be\label{eq:free-modules-SF}
H\ast\one  = H \quad ,\qquad H\ast\Pi\one = \Pi H \quad ,\qquad H\ast T\quad ,\qquad H\ast\Pi T\ .
\ee
By Theorem \ref{thm:H-unique}, $H$ is simple as a module over itself.
Similarly $\Pi H$ is a simple $H$-module (since a submodule $N\subset \Pi H$ gives rise to a submodule $\Pi N\subset H$).
In the following we will use Remark \ref{cfm} to prove:

\bth\lb{smsf}
The isomorphism classes of simple $H$-modules
	in $\SF(\h)$
are represented by 
$H$ and $\Pi H$.
\eth

The proof will be given after some more preparation.

\medskip

To compute cosocles of the last two free modules in \eqref{eq:free-modules-SF} we will use the isomorphisms
\be\label{eq:reduction-iso}
\begin{array}{rcccl}
	{_A\C}(A\ot X,A\ot Y) & \simeq &  \C(X,A\ot Y) & \simeq & \C(X\ot {^*Y},A) \\
	f &\mapsto& f \circ ( \eta\ot\id) &\mapsto& u(f) \ ,
\end{array}
\ee
where (we write out associators, but not unit constraints)
\begin{align}
	u(f) = \big[ 
	&X\ot {^*Y}, \xrightarrow{ \eta\ot\id}
	A \ot (X\ot {^*Y}) \xrightarrow{\alpha_{A,X,{^*Y}}}
	(A \ot X)\ot {^*Y}
	\nonumber \\ &
	 \xrightarrow{f\ot \id }
	(A \ot Y)\ot {^*Y}  \xrightarrow{\alpha_{A,Y,{^*Y}}^{-1}}
	A \ot (Y\ot {^*Y})  \xrightarrow{\id\ot ev_Y}
	A \big] \ .
\end{align}
For $X=Y$, the left hand side of \eqref{eq:reduction-iso} is an algebra, and the induced multiplication on the right hand side is the convolution product, $u(f \circ g) = \mu_A \circ (u(f) \otimes u(g)) \circ \Delta_{X\ot {^*X}}$. Here $\Delta_{X\ot {^*X}}$ denotes the Frobenius coproduct on $X\ot {^*X}$,
\begin{align}
	\Delta_{X\ot {^*X}} = \big[
	&X\ot {^*X}
	\xrightarrow{\id \otimes (\coev_X \otimes\id)}
	 X \otimes (({^*X}\ot X) \otimes {^*X}))
	\nonumber \\ &
	\xrightarrow{\id \otimes \alpha_{{^*X},X,{^*X}}^{-1}}
	 X \otimes ({^*X} \otimes (X \otimes {^*X}))
	\xrightarrow{\alpha_{X,{^*X},X\otimes {^*X}}}
	 (X \otimes {^*X}) \otimes (X \otimes {^*X})
	\big] \ .
	\label{eq:Delta_X*X}
\end{align}

Let us apply the above discussion to our situation, where $\C=\SF(\h)$ and $A=H$, for the choice $X=T$. The rigid structure on $\SF(\h)$ is discussed in \cite[Sect.\,3.6]{Davydov:2012xg}. We have ${^*T}=T$, so that for $f : H\ast T\to H\ast T$ we get a map $u(f) : T \ast T = S(\h) \to H$. Evaluating on $1 \in S(\h)$ gives an even element in $L_\mathrm{ev}$. Define $\tilde u(f) := u(f)(1) \in L_\mathrm{ev}$. 

\ble\lb{mfm}
The map $\tilde u : {_H\SF}(H\ast T, H\ast T) \to L_\mathrm{ev}$ is an isomorphism of algebras.
\ele 

\begin{proof}
The map $\SF_0(S(\h),L) \to L_\mathrm{ev}$, $f \mapsto f(1)$ is an isomorphism, and so $\tilde u$ is a composition of two isomorphism. 

To see that $\tilde u(f \circ g) = \tilde u(f) \Lmul \tilde u(g)$, we compute $\Delta_{{^*T} * T}$ from \eqref{eq:Delta_X*X}. We will need $\coev_T = \Lambda \in S(\h)$, where $\Lambda$ is the integral for $S(\h)$ defined in \cite[Lem.\,3.10]{Davydov:2012xg}. 
A short calculation using the associator in \eqref{eq:SF-assoc} shows that $\Delta_{{^*T}*T} : S(\h) \to S(\h) \otimes S(\h)$ is given by $a \mapsto \Delta(a) \, (\phi^{-1}(\Lambda)\otimes 1)$. 

The integral $\Lambda$ satisfies $\eps( \phi^{-1}(\Lambda))=1$, see \cite[Eqn.\,(3.61)]{Davydov:2012xg}. Since $S^0(\h) = \Cb \,1$, in our case this means $\phi^{-1}(\Lambda) = 1 \in S(\h)$. Altogether, $\Delta_{{^*T}*T}(1) = 1 \otimes 1$, proving that $\tilde u$ is an algebra map.
\end{proof}

\begin{proof}[Proof of Theorem \ref{smsf}]
Since $L_\mathrm{ev}$ has no idempotents other than $1_H$, Lemma \ref{mfm} shows that $H\ast T$ is indecomposable. Furthermore, $H\ast T$ is projective in ${_H\SF(\h)}$ (since $T$ is projective in $\SF(\h)$). An indecomposable projective module has simple cosocle. 
From
\be 
{_H\SF}(H\ast T, 
H)\ \simeq\ \SF(T, H)\ \simeq\ \SF_1(T, T) =  \Cb \id_T
\ee
it follows that the cosocle of $H\ast T$ must be $H$.
Similarly one shows that the cosocle of $H\ast \Pi T$ is $\Pi H$.
Finally,
\be
{_H\SF}(\Pi H, H)\ \simeq\ \SF(\Pi\one, H)\ \simeq\ \SF_0(\Pi\one, L) = 0\ ,
\ee
so that the $H$-modules $H$ and $\Pi H$ are not isomorphic.
\end{proof}

\bre
Since $\Pi H *_H \Pi H \cong H$,
the simple $H$-modules $H$ and $\Pi H$ form the group $\Zb/2\Zb$ with respect to the tensor product over $H$. Thus ${_H\SF(\h)}$ is an example of a 
pointed tensor category. 
In addition, one can verify that ${_H\SF(\h)}$ is not semi-simple.
\ere

We finish this section by showing that ${_H\SF(\h)}$ is finite tensor category.
We need the following auxiliary result:

\ble\lb{fcm}
Let $\C$ be a finite tensor category.
Let $A\in\C$ be an algebra such that any simple (left) $A$-module has a form $A\ot X$ for some (simple) $X\in\C$.
Then the category ${_A\C}$ is finite.
\ele
\bpf
Let $P\in\C$ be the projective cover of $X$. Then $A\ot P$ covers $A\ot X$ and is projective in ${_A\C}$.
Indeed the functor ${_A\C}(A\ot P,-) = \C(P,-)$ is exact. 
It is straightforward to see that an $A$-submodule which is an indecomposable direct summand of $A\ot P$, and which covers $A\ot X$, is the projective cover of $A\ot X$. 
\epf

\bco\lb{ftch}
The category ${_H\SF(\h)}$ is finite tensor.
\eco
\bpf
By Theorem \ref{smsf} simple $H$-modules are free. Lemma \ref{fcm} takes care of the rest.
\epf

\subsection{Simple local modules}\label{sec:simple-local-mod}

A commutative algebra $A$ in a braided tensor category is {\em Lagrangian} if any local $A$-module is  a direct sum of the regular $A$-module $A$.
Since submodules of a local module are local, Lagrangian algebras are in particular simple as modules over themselves. 

In the case of braided fusion categories, algebras with the above properties where introduced in \cite{Frohlich:2003hm} under the name ``trivialising algebras'' (see also \cite{dmno}, where the term ``Lagrangian" was introduced).

\bre
Recall that the left dual ${^*A}$ of an algebra is a left $A$-module with the module structure $A\ot {^*A}\to {^*A}$ given by
\be 
A\ot {^*A}\xrightarrow{coev \ot \id} {^*A}\ot A\ot A\ot {^*A} \xrightarrow{\id\ot \mu\ot \id} {^*A}\ot A\ot {^*A}\xrightarrow{\id\ot ev}  {^*A} \ .
\ee
It is straightforward to see that for commutative $A$ this module is local.
If $A$ is a Lagrangian algebra in $\C$ then the $A$-module ${^*A}$ is isomorphic to a direct sum of copies of $A$. 
But $({}^*A)^* \cong A$ as right $A$-modules, and since $A$ is simple as a left and right module over itself, ${}^*A$ must be isomorphic to a single copy of $A$. Altogether we have shown: a Lagrangian algebra in a braided tensor category is Frobenius. 
\ere

A commutative algebra $A$ in a ribbon tensor 
category is {\em Lagrangian} if it is Lagrangian as an algebra in a a braided category and if the ribbon twist is trivial on $A$, i.e. $\theta_A=1$.

For a local module $M$ in a ribbon category, the ribbon twist on $M$ is a module morphism. Together with $c_{V,U} \circ c_{V,U} = \theta_{U \otimes V} \circ (\theta_U^{-1} \otimes \theta_V^{-1})$, this observation implies the following lemma:

\ble\label{lem:twist-mult-id}
Let $A$ be a commutative algebra 
	with trivial twist
in a 
ribbon tensor category $\C$ 
	over an algebraically closed field.
Then a simple $A$-module $M$ is local if and only if the ribbon twist $\theta_M$ is a scalar multiple of the identity morphism $\id_M$.
\ele

From \eqref{eq:SF-twist} we see that $\theta_{\Pi H}$ is not a scalar multiple of the identity morphism: $\theta_{\Pi L} = \id_{\Pi L}$ while $\theta_{\Pi T} = -\id_{\Pi T}$.
Thus $H$ is the unique simple local $H$-module.
The main result of this section is:

\bth\lb{la}
The category of local $H$-modules ${^\mathrm{loc}_H\SF(\h)}$ is (canonically ribbon equivalent to) the category $\Vect$ of vector spaces.
In other words, the algebra $H$ is Lagrangian.
\eth

\bpf
We have seen that ${^\mathrm{loc}_H\SF(\h)}$ has only one simple object, the monoidal unit.
Since by Corollary \ref{ftch}, ${^\mathrm{loc}_H\SF(\h)}$ 
is a finite tensor category over a field of characteristic zero, by \cite[Thm.\,2.17]{eo} there are no non-trivial extensions of the unit by itself. Thus all objects of ${^\mathrm{loc}_H\SF(\h)}$ are direct sums of 
the monoidal unit. In other words ${^\mathrm{loc}_H\SF(\h)}$ is the category $\Vect$ of vector spaces.
\epf

The following theorem collects our results on Lagrangian algebras in $\SF(\h)$:

\begin{theorem}\label{thm:Lag-in-SF}
\begin{enumerate}
\item There exist Lagrangian algebras in $\SF(\h)$  iff $\beta=1$. 
\item For $\beta = 1$ and a Lagrangian subspace $\f\subset\h$, $H(\f)$ is a Lagrangian algebra.
\item A Lagrangian algebra  
$F \in \SF(\h)$ 
is isomorphic as an algebra to $H(\f)$ for some Lagrangian subspace $\f\subset\h$.
\item $H(\f) \cong H(\f')$ as algebras iff $\f = \f'$.
\item For $g \in Sp(\h)$ such that $g(\f) = \f'$ we have $G_{g^{-1}}(H(\f))\cong H(\f')$ as algebras, where $G_{g^{-1}}$ is the ribbon auto-equivalence from Section \ref{sec:brmon-on-rep}.
\end{enumerate}
\end{theorem}

\begin{proof}[Proof of Theorem \ref{thm:Lag-in-SF}]
1.--3.\ Let $F$ be a Lagrangian algebra in $\SF(\h)$. Write $F = F_0 \oplus F_1$ with $F_i \in \SF_i(\h)$. 
Suppose $F_1=\{0\}$. Then the $F$-modules $F$ and $\Pi F$ are local 
and non-isomorphic
	(since the self-braiding of $\Pi F$ in the category of local modules is $-\id_{\Pi F \otimes_F \Pi F}$), in contradiction to $F$ being Lagrangian.
Thus $F_1 \neq \{0\}$ and parts 1--3 follow from Theorems \ref{thm:H-unique} and \ref{la}. 

\smallskip\noindent
4.\ This is shown in Section \ref{sec:ex-unique}.

\smallskip\noindent
5.\ Abbreviate $J := G_{g^{-1}}$. Since $J$ is a ribbon functor, $J(H(\f))$ is again an algebra which satisfies W1--W5. 
By construction, the annihilator of $J(H(\f))_0$ in $\h$ is $\f'$. Thus, by Lemma \ref{lem:unique-from-f} we have $J(H(\f)) \cong H(\f')$ as algebras. 
\end{proof}

\bre
Here we outline an alternative proof of Theorem \ref{la}.
The Frobenius-Perron dimensions (with respect to the Grothendieck ring)
of the simple objects in $\SF$ are
$$d(\one) = d(\Pi\one) = 1,\qquad\qquad d(T) = d(\Pi T) = 2^\frac{d}{2}\ .$$
The dimension of the algebra $H$ is
$d(H) = d(L) + d(T) = 2^\frac{d}{2} + 2^\frac{d}{2} = 2^\frac{d+2}{2}$.
The dimension of the category $\SF_0$ is
$$D(\SF_0) = D(\sVect)d(S(\h)) = 2^{d+1}\ .$$
Here we define the dimension of a finite tensor category as the sum $\sum_S d(S)d(P(S))$ over (isomorphism classes of) simple objects of Frobenius-Perron dimensions multiplied by Frobenius-Perron dimensions of projective covers.
The dimension of the category $\SF$ is
$D(\SF) = D(\SF_0) + d(T)^2 + d(\Pi T)^2 = 2^{d+2}$.
Note that $D(\SF) = d(H)^2$. This together with the formula
\be\lb{dlm}
D({^\mathrm{loc}_A\C}) = \frac{D(\C)}{d(A)^2}
\ee
for a (Frobenius, simple) commutative algebra $A$ 
in a finite non-degenerate braided
tensor category $\C$ implies that $D({^\mathrm{loc}_H\SF(\h)})=1$ and hence ${^\mathrm{loc}_H\SF(\h)}=\Vect$.

We could not find a proof of \eqref{dlm} for finite
tensor categories in the literature 
(see \cite{Kirillov:2001ti} for the case of modular categories), 
and we will return to it in a separate publication.
\ere

\section{Classification of holomorphic extensions}

In this section we will combine the previous results to show that -- under two assumptions -- all holomorphic extensions of the even part $\VOA(\h)_\mathrm{ev}$ of the symplectic fermion VOSA are isomorphic to the lattice VOA in Corollary \ref{cor:symp-even-in-lattice}. To formulate the assumptions, we first need to review and extend some aspects of vertex operators for free super-bosons.

\subsection{Vertex operators for symplectic fermions}\label{sec:vertex-SF}

We start by reviewing the relevant aspects of vertex operators for free super-bosons as treated in \cite[Sect.\,3.1]{Runkel:2012cf}. For an object $C \in \SF(\h)$ we denote by $\overline{\widehat C}$ the algebraic completion of $\widehat C$ with respect to the $\Zb$-grading given by the total mode number, see \cite[Sect.\,2.2\,\&\,2.4]{Runkel:2012cf}.

\begin{definition} \label{def:vertex-op} \cite[Def.\,3.1]{Runkel:2012cf}~~For $r,s \in \Zb_2$ let $A \in \SF_r(\h)$,  $B \in \SF_s(\h)$ and $C \in \SF_{r+s}(\h)$. 
A {\em vertex operator from} $A,B$ {\em to} $C$ is a map
\be
	V ~:~ \Rb_{>0} \times (A \otimes \widehat B) \longrightarrow \overline{\widehat C}
\ee
such that 
\begin{enumerate}
\item[(i)] for all $x \in \Rb_{>0}$, $V(x)$ is an even linear map; 
for all $a\in A$, $\hat b \in \widehat B$, $\hat\gamma \in \widehat C'$ (the restricted dual), the function $x \mapsto \langle \hat \gamma , V(x)  (a \otimes \hat b) \rangle$ from $\Rb_{>0}$ to $\Cb$ is smooth;
\item[(ii)] $L_{-1} \circ V(x) - V(x)\circ (\id_A \otimes L_{-1}) = \tfrac{d}{dx} V(x)$, where $\tfrac{d}{dx} V(x)$ is defined in terms of matrix elements as in (i) and $L_{-1}$ was defined in \eqref{eq:virasoro-action};
\item[(iii)] for all $a \in \h$,
\begin{itemize}
\item $r=0$: For all $m \in \Zb$ (for $s=0$) or $m \in \Zb {+} \frac12$ (for $s=1$)
\be\label{eq:mode-past-untwisted}
  a_m \, V(x) = V(x)  \big( x^m  a \otimes \id + \id \otimes a_m \big)
\ee
\item $r=1$: For all $m \in \Zb$  (for $s=0$) or $m \in \Zb {+} \frac12$  (for $s=1$)
\be\label{eq:mode-past-twisted}
  \sum_{k=0}^\infty {{\frac 12} \choose k} (-x)^{k}  a_{m +\frac12- k}
     V(x) 
    = i 
 \sum_{k=0}^\infty {{\frac 12} \choose k} (-1)^k x^{\frac12-k}  
       V(x) \big( \id \otimes a_{m + k} \big) \ .
\ee
\end{itemize}
\end{enumerate}
The vector space of all vertex operators from $A,B$ to $C$ will be denoted by $\Vc_{A,B}^C$.
\end{definition}

Note that the vertex operators defined above are not the same as intertwining operators for representations of a VOA (see \cite{Milas:2001,Huang:2010} for the definition of logarithmic intertwining operators): they are not defined in terms of a formal variable and its formal logarithm;  they only satisfy a subset of the conditions of intertwining operators; they are defined only on $A \otimes \widehat B$ instead of $\widehat A \otimes \widehat B$. 
The relation between vertex operators and intertwining operators is part of Conjecture \ref{conj:braided-equiv} below.

\medskip

By definition of the tensor product $*$ of $\SF(\h)$ we have natural
isomorphisms \cite[Thm.\,3.13]{Runkel:2012cf}
\be\label{eq:SF-tensor-via-rep-obj}
	\SF(A*B,C)
	\longrightarrow \Vc_{A,B}^C
	\quad , \quad
	f \longmapsto V(f,-) \ .
\ee
The vertex operator $V(f,x)$ is uniquely defined by its value on ground states. 
Namely, write $P_\mathrm{gs} : \widehat A \to A$ for the projector to ground states. 
Let $\Omega$ be the copairing dual to $(-,-)$, that is, in terms of two bases $\{ \alpha^i \}$ and $\{ \beta^i \}$ such that $(\alpha^i,\beta^j)=\delta_{i,j}$ we have $\Omega = \sum_{i=1}^d \beta^i \otimes \alpha^i \in S(\h) \otimes S(\h)$. Furthermore, set $\hat\Omega = \sum_{i=1}^d \beta^i \alpha^i \in S(\h)$.
Then, depending on the sector $A$ and $B$ are taken from, we have \cite[Eqn.\,(3.62)]{Runkel:2012cf}
\be\label{eq:Pgs-V}
P_\mathrm{gs} \circ V(f;x)\raisebox{-2pt}{$\big|_{A \otimes B}$} ~=~  
\left\{\rule{0pt}{3.6em}\right.
\hspace{-.5em}\raisebox{.7em}{
\begin{tabular}{ccl}
  $A$ & $B$ & 
\\
\rm 0 & 0 & 
$f \circ \exp\!\big(\ln(4x)\, \Omega \big)$ 
\\[.5em]
 0 & 1  & 
$f \circ \exp\!\big(- \tfrac{1}{2} \ln(16x)  \, (\hat \Omega \otimes \id) \big)$ 
\\[.5em]
1 & 0  & 
$f \circ \exp\!\big(- \tfrac{1}{2} \ln(16x) \,  (\id \otimes \hat \Omega)\big)$ 
\\[.5em]
1 & 1  &
$\exp\!\big(\tfrac12\ln(x)\,(\tfrac{d}4 + \hat \Omega)\big) \circ f \circ (1 \otimes \id_{A \otimes B})$ 
\end{tabular}}
\ee
Here, the ``$1$'' in the expression in sector 11 stands for the unit of $S(\h)$.

\medskip

Below we will extend the vertex operators to maps $\Rb_{>0} \times (\widehat A \otimes \widehat B) \to \overline{\widehat C}$, at least in the case $A \in \SF_0(\h)$. The extension rules
(as well as the defining rules \eqref{eq:mode-past-untwisted} and \eqref{eq:mode-past-twisted})
are obtained via the usual contour deformation arguments, see 
\cite{Gaberdiel:1993mt,Gaberdiel:1996kf}
and e.g.\ \cite[App.\,B.3]{Runkel:2012cf}, which also states our conventions for square root branch cuts. 

For the purpose of the following proposition we denote the extension by $\widehat V$, but afterwards we will just write $V$.
The proof of the proposition is a little lengthy and has been moved to Appendix \ref{app:V-on-descendents}.

\begin{proposition}\label{prop:V-on-descendents}
Let $A \in \SF_0(\h)$, $B,C \in \SF_s(\h)$ for $s \in \Zb_2$ . For each $V \in \mathcal{V}_{AB}^C$ as in Definition \ref{def:vertex-op} there exists a unique
\be
	\widehat V ~:~  \Rb_{>0} \times (\widehat A \otimes \widehat B) \longrightarrow \overline{\widehat C}
\ee
such that $\widehat V(x)(a \otimes \hat b) = V(x)(a \otimes \hat b)$ for all $a \in A$ and $\hat b \in \widehat B$ and for all $m \in \Zb$,
\begin{itemize}
\item for $s=0$:
\begin{align}
&V(x)(a_m \otimes \id)
\nonumber\\
&=
\sum_{k=0}^\infty {m \choose k} (-x)^k \, a_{m-k} V(x)
\,-\,
\sum_{k=0}^\infty {m \choose k} e^{\pi i(m-k)} x^{m-k} \, V(x)(\id \otimes a_k)
\label{eq:mode-at-x_untw0}
\end{align}
\item for $s=1$:
\begin{align}
&\sum_{k=0}^\infty {\tfrac12 \choose k} x^{\frac12 -k} \,
V(x)(a_{m+k} \otimes \id)
\nonumber\\
&=
\sum_{k=0}^\infty {m \choose k} (-x)^k \, a_{m+\frac12-k} V(x)
\,-\,
\sum_{k=0}^\infty {m \choose k} e^{\pi i(m-k)} x^{m-k} \, V(x)(\id \otimes a_{k+\frac12})
\label{eq:mode-at-x_tw0}
\end{align}
\end{itemize}
This $\widehat V$ satisfies (i) in Definition \ref{def:vertex-op}  (for all $\hat a \in \widehat A$).
\end{proposition}

We remark that for $\h$ purely even (rather than purely odd as treated here), twisted and untwisted vertex operators -- in the case that the module inserted at $x$ is untwisted -- are defined in \cite{Frenkel:1987}, see e.g.\ (3.2.7), (4.1.8) there, as well as Chapters 8, 9. 

\medskip

From now on we will refer to $\widehat V$ just as $V$.

\subsection{Vertex operators and symplectic maps}

An element $g \in Sp(\h)$ induces automorphisms $\hat g$ of the Lie super-algebras $\widehat\h$ and $\widehat\h_{\mathrm{tw}}$. Via pullback by $\hat g^{-1}$ we obtain an $Sp(\h)$ action on 
\be\label{eq:reps-untw-and-tw}
	\Rep^\mathrm{fd}_{\flat,1}(\widehat\h)
	~~\oplus~~
	\Rep^\mathrm{fd}_{\flat,1}(\widehat\h_\mathrm{tw})
\ee
by auto-equivalences. Denote this functor by $\widehat G_g$.

Recall the ribbon auto-equivalence $G_g$ defined in the end of Section \ref{sec:brmon-on-rep}. The functors $\widehat{(-)} \circ G_g$ and $\widehat G_g \circ \widehat{(-)}$ from $\SF(\h)$ to \eqref{eq:reps-untw-and-tw} are naturally equivalent via
\be\label{eq:nat-xfer-hatG-Ghat}
\varphi_{g,A} \,:=\,
U(\hat g) \otimes \id ~:~ \widehat G_g(\widehat A) \longrightarrow \widehat{G_g(A)} \ ,
\ee
for all $A \in \SF(\h)$, where we make use of the form \eqref{eq:induced-mod-def} of the induced module $\widehat A$. The automorphism $U(\hat g)$ ensures that this is a map of $\widehat\h_{(\mathrm{tw})}$-modules 
and that it is compatible with the tensor product over $\widehat\h_{(\mathrm{tw}),\ge0} \oplus \Cb K$ used in \eqref{eq:induced-mod-def}.

The isomorphisms \eqref{eq:SF-tensor-via-rep-obj} now fit into a commuting square:

\begin{lemma}\label{lem:Gg-functor-on-vertex-op}
For all $A,B,C \in \SF(\h)$ we have the commuting square
\be
\xymatrix{
\SF(A * B,C) \ar[rr] \ar[d]_{(1)} && \Vc_{A,B}^C \ar[d]^{(2)}
\\ 
\SF(G_g(A) * G_g(B),G_g(C)) \ar[rr] && \Vc_{G_g(A),G_g(B)}^{G_g(C)}
}
\ee
where (1) is the map $f \mapsto G_g(f) \circ (G_g)_{A,B}^{-1}$ and (2) is the map $V(x) \mapsto \varphi_{g,C} \circ V(x) \circ (\varphi_{g,A}^{-1}\otimes \varphi_{g,B}^{-1})$.
\end{lemma}

In writing map (2) above, we used that the underlying $\Zb$-graded vector spaces of $\widehat G_g(\widehat D)$ and $\widehat D$ are the same for all $D \in \SF(\h)$.

\begin{proof}
Pick $f : A*B \to C$, let $V = V(f,-)$ and write $V'(x) = \varphi_{g,C} \circ V(x) \circ (\varphi_{g,A}^{-1}\otimes \varphi_{g,B}^{-1})$.

One first verifies that $V(x)$ is equally a vertex operator 
$G_g(A) \otimes \widehat G_g(\widehat B) \to \overline{\widehat G_g(\widehat C)}$, i.e.\ it satisfies the conditions in Definition \ref{def:vertex-op} for these modules with $g$-twisted $\widehat\h_{(\mathrm{tw})}$-action (and, if $A \in \SF_0(\h)$, the conditions of Proposition \ref{prop:V-on-descendents}). Since $\varphi_{g,D}$ are $\widehat\h_{(\mathrm{tw})}$-module isomorphisms for all $D \in \SF(\h)$ we get
\be
	V' \,\in\, \Vc_{G_g(A),G_g(B)}^{G_g(C)} \ .
\ee

Therefore, $V'$ is determined by its values on ground states via \eqref{eq:Pgs-V}. 
But on ground states $v \in D$ we have $\varphi_{g,D}(1 \otimes v) = 1 \otimes v$, and so it remains to check that 
\be
	P_{\mathrm{gs}} \circ V'(x)\raisebox{-2pt}{$\big|_{A \otimes B}$} = P_{\mathrm{gs}} \circ V\big(\,G_g(f) \circ (G_g)_{A,B}^{-1}\,,\,x\big)\raisebox{-2pt}{$\big|_{A \otimes B}$} \ ,
\ee
which is immediate from the definitions
	(both sides are equal to $P_{\mathrm{gs}} \circ V(x)|_{A \ot B}$).
\end{proof}

\subsection{Two assumptions}\label{sec:two-assumptions}

The VOA $\VOA(\h)_\mathrm{ev}$ is $C_2$-cofinite  \cite{Abe:2005} and according to \cite{Huang:2010} the representation category $\Rep(\VOA(\h)_\mathrm{ev})$ is a braided tensor category. We conjecture that $\Rep(\VOA(\h)_\mathrm{ev})$ is equivalent as a braided tensor category to $\SF(\h)$ for $\beta = e^{- \pi i d/8}$. To state this conjecture precisely, we will split it into two parts. The first part we can state now, while the second requires some additional notation.

\begin{conjecture}\label{conj:braided-equiv} {\bf (part 1)}
For each $X \in \SF(\h)$, $({\widehat X})_\mathrm{ev}$ is a (logarithmic) $\VOA(\h)_\mathrm{ev}$-module. 
The functor
\be
F ~:~ \SF(\h)\longrightarrow\Rep(\VOA(\h)_\mathrm{ev})\quad ,\quad X\longmapsto ({\widehat X})_\mathrm{ev} \ ,
\ee 
is a $\Cb$-linear equivalence. 
\end{conjecture}

For $U,V,W \in \Rep(\VOA(\h)_\mathrm{ev})$ denote by $\tilde{\mathcal{V}}^U_{V,W}$ the space of logarithmic intertwining operators of type ${U \choose V \, W}$, see \cite[Def.\,3.14]{Huang:2010}, and by $\tilde\otimes$ the tensor product bifunctor of $\Rep(\VOA(\h)_\mathrm{ev})$ given in \cite[Def.\,4.15]{Huang:2010}.
Then $V \tilde\otimes W$ is a representing object for $U \mapsto \tilde{\mathcal{V}}^U_{V,W}$:
\be\label{eq:HLZ-tensor-def}
	\nu \,:\,
	\Hom_{\VOA(\h)_\mathrm{ev}}(V \tilde\otimes W,U)  
\xrightarrow{~~\sim~~} 
\tilde{\mathcal{V}}^U_{V,W} \ ,
\ee
naturally in $U$.\footnote{In \cite{Huang:2010}[Def.\,4.15] the tensor product is defined via an initial object condition, but this definition is equivalent to the one in \eqref{eq:HLZ-tensor-def}.}

For $A,B,C \in \SF(\h)$, consider a vertex operator $V \in \mathcal{V}_{A,B}^C$ as in Definition \ref{def:vertex-op}. We saw in Proposition \ref{prop:V-on-descendents} that for $A \in \SF_0(\h)$, $V$ has a unique extension from $A$ to $\widehat A$. Conjecturally, such a unique extension (with an appropriate change of \eqref{eq:mode-at-x_untw0} and \eqref{eq:mode-at-x_tw0}) is possible also for $A \in \SF_1(\h)$. 

\medskip\noindent
{\bf Conjecture \ref{conj:braided-equiv}.} {\bf (part 2)}
For each $A,B,C \in \SF(\h)$, $V \in \mathcal{V}_{A,B}^C$ has a unique extension to $V : \Rb_{>0} \times (\widehat A \otimes \widehat B)$. For each such $V$ there exists a unique intertwining operator $\tilde V \in \tilde{\mathcal{V}}^{F(C)}_{F(A),F(B)}$, whose formal expression converges to a smooth function on $\Rb_{>0}$ and which agrees with $V$, restricted to even subspaces. This induces an isomorphism
\be
	f_{A,B}^C \,:\, \mathcal{V}_{A,B}^C \longrightarrow \tilde{\mathcal{V}}^{F(C)}_{F(A),F(B)} \ .
\ee
natural in $A,B,C \in \SF(\h)$. The resulting isomorphism $F(A) \tilde\otimes F(B) \to F(A * B)$ turns $F$ into a braided monoidal functor.

\medskip

The isomorphism $F(A) \tilde\otimes F(B) \to F(A * B)$ is obtained by noting that $\SF(A*B,C)$ and $\Hom_{\VOA(\h)_\mathrm{ev}}(F(A*B),F(C))$ are naturally isomorphic by part 1. By \eqref{eq:SF-tensor-via-rep-obj} and \eqref{eq:HLZ-tensor-def} we obtain a natural isomorphism $\Hom_{\VOA(\h)_\mathrm{ev}}(F(A*B),F(C)) \to \Hom_{\VOA(\h)_\mathrm{ev}}(F(A) \tilde\otimes F(B),F(C))$. 

\medskip

Parts 1 and 2 of Conjecture \ref{conj:braided-equiv} will be our first assumption in Theorem \ref{thm:classify-extensions} below. Next we describe the second assumption, which extends results obtained in \cite{Huang:2014ixa}. 
The following theorem is shown in \cite{Huang:2014ixa} (see Theorem 3.2, Remark 3.3, and the last paragraph of the Appendix there).

\begin{theorem}\label{thm:HKL}
Let $\VOA$ be a simple $C_2$-cofinite VOA.
\begin{enumerate}
\item
Let $A$ be a commutative (unital, associative) algebra in $\Rep\VOA$ with trivial twist, and denote by $\iota: \one \to A$ and $\mu:A \tilde\otimes A \to A$ its unit and multiplication. Then
$A$, together with $\iota : \VOA \to A$ and $\nu(\mu) \in \tilde{\mathcal{V}}^{A}_{A,A}$ (recall \eqref{eq:HLZ-tensor-def}) as state-field map, is a VOA extending $\VOA$. Denote this extension by $\VOA_A$.
\item
The assignment $A \mapsto \VOA_A$ defines a functorial equivalence from the category of 
\begin{quote}
commutative algebras in $\Rep\VOA$ with trivial twist,  and with algebra homomorphisms as morphisms
\end{quote}
to the category of 
\begin{quote}
extensions of $\VOA$ and VOA homomorphisms, commuting with the embedding of $\VOA$. 
\end{quote}
\end{enumerate}
\end{theorem}

In addition we need to assume the following
statement, which is shown in \cite{Huang:2014ixa}  for rational VOAs subject to some additional technical conditions.
\begin{itemize}
\item[\bf (A)]
Let $M \in \Rep\VOA$ be a local $A$-module with action $\rho : A \tilde\otimes M \to M$.
The assignment $(M , \rho) \mapsto (M , \nu(\rho))$ (recall \eqref{eq:HLZ-tensor-def}) gives a $\Cb$-linear equivalence between the category of local $A$-modules and logarithmic $\VOA_A$-modules.
\end{itemize}
For our application we require Assumption (A) only for $\VOA = \VOA(\h)_\mathrm{ev}$.

\subsection{Classification}

Fix a finite dimensional symplectic space $\h$ and abbreviate $\VOA = \VOA(\h)_\mathrm{ev}$.

\begin{lemma}\label{lem:class1}
Suppose Conjecture \ref{conj:braided-equiv} and Assumption (A) hold. Let $\mathbb{U},\mathbb{U}'$ be holomorphic extensions of $\VOA$. Then
\begin{enumerate}
\item
there exists a Lagrangian subspace $\f \subset \h$ such that
$\mathbb{U}\cong\VOA_{F(H(\f))}$ as extensions (i.e.\ via an isomorphism of VOAs that commutes with the embedding of $\VOA$ into $\mathbb{U}$ and $\VOA_{F(H(\f))}$). Dito for $\f'$ and $\mathbb{U}'$.
\item
$\mathbb{U}$ and $\mathbb{U}'$ are isomorphic as extensions iff $\f = \f'$.
\end{enumerate}
\end{lemma}

\begin{proof}
1.\ By Theorem \ref{thm:HKL}, there is a commutative algebra $A \in \Rep\VOA$ with trivial twist s.t.\ $\mathbb{U} = \VOA_A$. Since $\mathbb{U}$ is a holomorphic extension, by Assumption (A) every simple local $A$-module is isomorphic to $A$. 
Hence, up to isomorphism, the tensor unit $A$ is the only simple object in the finite tensor category of local $A$-modules (cf.\ Lemma \ref{fcm}
	-- $\Rep\VOA$ is finite by Conjecture \ref{conj:braided-equiv} as $\SF(\h)$ is finite).
It then follows from \cite[Thm.\,2.17]{eo} that this category is equivalent to $\vect$. 
Thus $A$ is Lagrangian and by Theorem \ref{thm:Lag-in-SF} and Conjecture \ref{conj:braided-equiv} there is a Lagrangian subspace $\f \subset \h$ such that $A \cong F(H(\f))$ as algebras.

\smallskip\noindent
2.\ Let $\mathbb{U}' = \VOA_{A'}$.
The existence of an isomorphism of extensions $\mathbb{U} \cong \mathbb{U}'$ is equivalent to the existence of an isomorphism of algebras $A \cong A'$ in $\Rep\VOA$ (Theorem \ref{thm:HKL}). This in turn is equivalent to an isomorphism of algebras between $H(\f)$ and $H(\f')$, i.e.\ to $\f=\f'$ (Theorem \ref{thm:Lag-in-SF}).
\end{proof}

\begin{lemma}\label{lem:class2}
Suppose Conjecture \ref{conj:braided-equiv} holds.
Let $\f$ and $\f'$ be two Lagrangian subspaces of $\h$. Then the VOAs $\VOA_{F(H(\f))}$ and $\VOA_{F(H(\f'))}$ are isomorphic as VOAs.
\end{lemma}

An isomorphism of VOAs preserves the stress tensor, but for $\f \neq \f'$ the above isomorphism will not be compatible with the embedding of $\VOA$, i.e.\ it is not an isomorphism of extensions.

\begin{proof}
Let $H = H(\f)$, $H' = H(\f')$, and denote by $\mu$ and $\mu'$ the multiplication on $H$ and $H'$, respectively. Write $V = V(\mu,-)$ and $V' = V(\mu',-)$ for the vertex operators assigned to $\mu$ and $\mu'$ via \eqref{eq:SF-tensor-via-rep-obj}.

After invoking Conjecture \ref{conj:braided-equiv}, to prove the lemma we need to find a linear isomorphism $\varphi : (\widehat H)_\mathrm{ev} \to (\widehat H')_\mathrm{ev}$ which
\begin{enumerate}
\item preserves the $\mathbb{Z}$-grading,
\item maps $1_H$ to $1_{H'}$,
\item preserves the stress tensor,
\item satisfies $\varphi \circ V(x) = V'(x) \circ (\varphi \otimes \varphi)$.
\end{enumerate}
Pick $g \in Sp(\h)$ such that $g(\f)=\f'$ and abbreviate
 $J := G_{g}$
	(as defined in the end of Section \ref{sec:brmon-on-rep}).  
From Theorem \ref{thm:Lag-in-SF} we know that $J(H) \cong H'$ as algebras. Without loss of generality we can therefore assume that $J(H) = H'$ as algebras.

Define $\varphi$ to be the isomorphism $\varphi_{g,H} : \widehat J(\widehat H) \to \widehat{J(H)}$ from \eqref{eq:nat-xfer-hatG-Ghat}, restricted to even subspaces. Note that as $\mathbb{Z}$-graded vector spaces $\widehat J(\widehat H) = \widehat H$, and so $\varphi$ maps between the correct spaces (but it does not intertwine the $\widehat\h_{(\mathrm{tw})}$-action).

Property 1 is clear, property 2 follows from naturality of $\varphi_{g,-}$ (see \eqref{eq:nat-xfer-hatG-Ghat}), and for property 3 note that the element \eqref{eq:SF-stress-tensor} only depends on the symplectic structure of $\h$. Finally, property 4 is Lemma \ref{lem:Gg-functor-on-vertex-op}.
\end{proof}

Lemmas \ref{lem:class1} and \ref{lem:class2}, together with Corollary \ref{cor:symp-even-in-lattice}, establish the following theorem:

\begin{theorem}\label{thm:classify-extensions}
Suppose Conjecture \ref{conj:braided-equiv} and Assumption (A) hold. 
Then $\VOA=\VOA(\h)_\mathrm{ev}$ has holomorphic extensions iff $d \in 16\Zb_{>0}$. For $d \in 16\Zb_{>0}$,
\begin{enumerate}
\item the assignment $\f \mapsto \VOA_{F(H(\f))}$ gives a bijection between Lagrangian subspaces of $\h$ and isomorphism classes of holomorphic extensions of $\VOA$ (i.e.\ extensions modulo isomorphisms of VOAs commuting with the embedding of $\VOA$),
\item all holomorphic extensions of $\VOA$ are isomorphic as VOAs to $\VOA_{D_{d/2}^+}$ with stress tensor as in Theorem \ref{thm:symp-ferm-in-lattice}.
\end{enumerate}
\end{theorem}

\appendix

\section{Proof of Proposition \ref{prop:V-on-descendents}}\label{app:V-on-descendents}

Equations \eqref{eq:mode-at-x_untw0} and \eqref{eq:mode-at-x_tw0} can be used to reduce $\widehat V(x)$ to $V(x)$, so that
uniqueness and property (i) from Definition \ref{def:vertex-op} are immediate. 
Existence is more involved: In Lemma \ref{lem:extension-aux1}, we establish a ``coproduct property'' analogous to \cite{Gaberdiel:1993mt,Gaberdiel:1996kf}. We then proceed in three steps to show existence of $\widehat V(x)$. 

\medskip

We will use a constant $\delta$ which takes the value $\delta=0$ for $B,C \in \SF_0(\h)$ and $\delta = \frac12$ for $B,C \in \SF_1(\h)$. In the case $\delta=\frac12$, we will need the auxiliary Lie super-algebra $\widetilde\h$, which has the same basis as $\widehat\h$ but the (symmetric) bracket is given by
\be
	[a_m,b_n]
	~=~ (a,b) \, \big(
	m   x \delta_{m+n,0}
	+
	\big(m+\tfrac12\big) \delta_{m+n+1,0} \big) \, K \ ,
\ee
for some fixed $x \in \Rb_{>0}$.
A similar Lie super-algebra was used in \cite[App.\,A.3]{Runkel:2012cf} to treat twisted sector vertex operators. We will abbreviate
\be
	\h^\delta = \begin{cases} ~\widehat\h &:~ \delta = 0 \\ ~\widetilde\h &: ~\delta=\tfrac12 \end{cases} \qquad .
\ee

For $a_m \in \h^\delta$ consider the endomorphisms $\rho^\delta(a_m)$ of $\widehat B$ and $\lambda^\delta(a_m)$ of $\overline{\widehat C}$ defined as follows:
\be\label{eq:lamrho-modes-def}
\rho^\delta(a_m) = -
\sum_{k=0}^\infty {m \choose k} e^{\pi i(m-k)} x^{m-k} \, a_{k+\delta}
\quad , \quad
\lambda^\delta(a_m) = \sum_{k=0}^\infty {m \choose k} (-x)^k \, a_{m+\delta-k}  
\ ,
\ee
The infinite sum in $\lambda^\delta(a_m)$ does not truncate, but it only has finitely many contributions to any given graded component of $\overline{\widehat C}$. Define the endomorphism $\xi^\delta$ of $\widehat A$ as
\be\label{eq:xi-modes-def}
	\xi^0(a_m) = a_m \quad , \quad
	\xi^{\frac12}(a_m) = \sum_{k=0}^\infty {\tfrac12 \choose k} x^{\frac12 -k} \, a_{m+k} \ .
\ee
With these notations, conditions \eqref{eq:mode-at-x_untw0} and \eqref{eq:mode-at-x_tw0}  become simply
\be\label{eq:extension-condition-aux1}
	\widehat V(x) \big(\xi^\delta(a_m) \otimes \id\big)
	\,=\, \lambda^\delta(a_m) \widehat V(x) \,+\, \widehat V(x) \big(\id \otimes \rho^\delta(a_m)\big) \ .
\ee

\begin{lemma}\label{lem:extension-aux1}
$\xi^\delta$ and $\lambda^\delta$ define representations of $\h^\delta$ with $K=1$, while $\rho^\delta$ defines a representation of $\h^\delta$ with $K=0$.
\end{lemma}

\begin{proof}
We treat the case $\delta=\frac12$ only. The case $\delta=0$ works along the same lines but is simpler. Firstly, $[\rho^{\frac12}(a_m),\rho^{\frac12}(b_n)]=0$ is clear as only positive modes appear. Next, inserting the definition of $\xi^{\frac12}$ gives
\be
[\xi^{\frac12}(a_m),\xi^{\frac12}(b_n)]
~=~ (a,b)
x^{m+n+1} C_{m,n} 
\quad , \quad C_{m,n} = \sum_{k,l=0}^\infty {\frac12 \choose k}{\frac12 \choose l} \, (m+k) \, \delta_{m+n+k+l,0} \ .
\ee
To evaluate such sums, it is best to rewrite them as residues. In the present case, 
\be
	C_{m,n} = \mathrm{Res}_0\Big( x^n (1+x)^{\frac12}
	\,\tfrac{d}{dx}\big\{ x^m (1+x)^{\frac12} \big\} \Big) 
	= m \delta_{m+n,0} + (m+\tfrac12) \delta_{m+n+1,0} \ .
\ee
Thus $[\xi^{\frac12}(a_m),\xi^{\frac12}(b_n)] = \xi^{\frac12}([a_m,b_n)]$, as required.
For  $\lambda^\delta$ the corresponding calculation is
$[\lambda^{\frac12}(a_m),\lambda^{\frac12}(b_n)]
= (a,b)
(-x)^{m+n+1} D_{m,n}$,
where
\begin{align}
	D_{m,n} &= \sum_{k,l=0}^\infty {m \choose k}{n \choose l} \, (m+\tfrac12-k) \, \delta_{m+n+1-k-l,0} 
	\nonumber\\
	&= -\mathrm{Res}_0\Big( 
	x^{-n-\frac12} (1+x)^{n} \,
	\tfrac{d}{dx}\big\{ x^{-m-\frac12} (1+x)^{m} \big\} \Big) 
	\nonumber\\
	&= 
	(m+\tfrac12) \, \mathrm{Res}_0\big( x^{-m-n-2}(1+x)^{m+n-1} \big)
	\,+\, \tfrac12 \,\mathrm{Res}_0\big( x^{-m-n-1}(1+x)^{m+n-1} \big)
	\nonumber\\
	&= 
	 -m \,\delta_{m+n,0} + (m+\tfrac12)\, \delta_{m+n+1,0} \ .
	 \label{eq:extend-aux3}
\end{align}
To evaluate the residues in the last step, one can proceed by distinguishing cases depending on whether the corresponding contour integral is more easily evaluated around $0$ or $\infty$. Inserting \eqref{eq:extend-aux3} into $(a,b)
(-x)^{m+n+1} D_{m,n}$ gives the required answer.
\end{proof}

Define the actions $N_x^\delta(a_m)$ and $M_x^\delta(a_m)$ on $\Hom(\widehat A \otimes \widehat B , \overline{\widehat C})$ by setting, for an element $F$ in the Hom-space,
\be
N_x^\delta(a_m)(F) = F\circ(\xi^\delta(a_m) \otimes \id)
\quad , \quad
M_x^\delta(a_m)(F) = \lambda^\delta(a_m)\circ F + F \circ (\id \otimes \rho^\delta(a_m)) \ .
\ee
Then \eqref{eq:extension-condition-aux1} simplifies further to $N_x^\delta(a_m)(\widehat V(x)) = M_x^\delta(a_m)(\widehat V(x))$. Below, it will be important that Lemma \ref{lem:extension-aux1} implies
\be\label{eq:extension-condition-aux2}
	\big[ N_x^\delta(a_m) , N_x^\delta(b_m) \big] ~=~ N_x^\delta\big([a_m,b_n]\big) 
	\quad , \quad
	\big[ M_x^\delta(a_m) , M_x^\delta(b_m) \big] ~=~ M_x^\delta\big([a_m,b_n]\big) \ ,
\ee
where the brackets are those of $\h^\delta$. 

We now prove existence of $\widehat V(x)$ in three steps.

\medskip\noindent
{\em Step 1: Fixing $\widehat V(x)$.} Pick an ordered basis $\{\alpha^i\}_i$ of $\h$. Use  \eqref{eq:extension-condition-aux1} to define $\widehat V(x)$ on the corresponding ``PBW-like'' basis of $\widehat A$ obtained via $\xi^\delta$, that is, on vectors of the form $\xi^\delta(\alpha^{i_1}_{-m_1}) \cdots \xi^\delta(\alpha^{i_n}_{-m_n}) a$, where $a \in A$, $m_1 \ge \cdots \ge m_n \ge 1$ and $i_k>i_{k+1}$ in case $m_k = m_{k+1}$. 
Indeed, for $\delta=0$, this is just the PBW-basis itself, while for $\delta=\frac12$ one can see by induction on the total mode number that this is again a basis.
Basis-independence will follow from the already-established uniqueness, once we have proved that the so-defined $\widehat V(x)$ satisfies \eqref{eq:extension-condition-aux1} for all $a\in \h$ and $m \in \Zb$. 

\medskip\noindent
{\em Step 2: The ground state condition.} For $a \in \h$, $u \in A$ and $m>0$ we have $a_m u=0$. To establish \eqref{eq:extension-condition-aux1} in this case, we need to show that
\be
0=
\sum_{k=0}^m {m \choose k} (-x)^k \, a_{m+\delta-k} V(x)
\,-\,
\sum_{k=0}^m {m \choose k} (-x)^{k} \, V(x)(\id \otimes a_{m+\delta-k})
\ee
holds on $A \ot \widehat{B}$.
Here, we used that for $m \ge 0$, the sum truncates,  we replaced $k \leadsto m-k$ in the second sum, and we applied the symmetry of the binomial coefficient. Applying \eqref{eq:mode-past-untwisted} reduces the above equation to \be
	0 = \sum_{k=0}^m {m \choose k} (-1)^k \,V(x) (x^{m+\delta} a \otimes \id) \ ,
\ee
which evidently holds.

\medskip\noindent
{\em Step 3: Condition \eqref{eq:extension-condition-aux1} in general.} We proceed by induction on the total number $N$ of modes in the argument from $\widehat{A}$. 

\medskip\noindent
$N=0$: In this case, $\widehat V(x)$ is evaluated on a ground state $u \in A$. For $m \le 0$, condition \eqref{eq:extension-condition-aux1} holds by definition of $\widehat V(x)$, and for $m>0$ the condition was established in step 2.

\medskip\noindent
$N \leadsto N+1$: Let $\hat u \in \widehat A$ be a basis element in the basis used to define $\widehat V(x)$. We consider the situation $\widehat V(x) (\xi^\delta(\alpha^i_{m}) \hat u) \otimes \hat v$, where $\hat u = \xi^\delta(\alpha^j_{-n}) \hat u'$ for some $j,n$ and $ \hat u'$ of grade $N-n$, and $\hat v \in \widehat B$ arbitrary. 
\begin{itemize}
\item
For $m<-n$, or for $m=-n$ and $i>j$, condition \eqref{eq:extension-condition-aux1} holds by construction, as $\xi^\delta(\alpha^i_{m}) \xi^\delta(\alpha^j_{-n}) \hat u'$ is part of the defining basis.
\item
For $m=-n$ and $i<j$ we need to move $\xi^\delta(\alpha^i_m)$ to its appropriate position to obtain a defining basis element. Let us illustrate the argument in the case where only one exchange is necessary, i.e.\ when  $\xi^\delta(\alpha^j_{-n}) \xi^\delta(\alpha^i_m) \hat u'$ is an element of the defining basis. The general case is an iteration of this argument.
Writing $\widehat V$ instead of $\widehat V(x)$ to reduce the number of brackets, we get
\begin{align}
\big(N_x^\delta(\alpha^i_m)\widehat V\big)(\hat u \otimes \hat v) 
&\overset{(1)}= 
\big(N_x^\delta(\alpha^i_m)N_x^\delta(\alpha^j_{-n})\widehat V\big)(\hat u' \otimes \hat v) 
\nonumber\\
&\overset{(2)}= 
-\big(N_x^\delta(\alpha^j_{-n})N_x^\delta(\alpha^i_m)\widehat V\big)(\hat u' \otimes \hat v) 
\nonumber\\
&\overset{(3)}= 
-\big(M_x^\delta(\alpha^j_{-n})M_x^\delta(\alpha^i_m)\widehat V\big)(\hat u' \otimes \hat v) 
\nonumber\\
&\overset{(4)}= 
\big(M_x^\delta(\alpha^i_m)M_x^\delta(\alpha^j_{-n})\widehat V\big)(\hat u' \otimes \hat v) 
\nonumber\\
&\overset{(5)}= 
\big(M_x^\delta(\alpha^i_m)N_x^\delta(\alpha^j_{-n})\widehat V\big)(\hat u' \otimes \hat v) 
\nonumber\\
&\overset{(6)}= 
\big(M_x^\delta(\alpha^i_m)\widehat V\big)(\hat u \otimes \hat v) \ .
\end{align}
In step 1, $\hat u = \xi^\delta(\alpha^j_{-n}) \hat u'$ has been inserted and step 2 uses \eqref{eq:extension-condition-aux2}. In step 3 we use the assumption that $\xi^\delta(\alpha^j_{-n}) \xi^\delta(\alpha^i_m) \hat u'$ is an element of the defining basis, and so \eqref{eq:extension-condition-aux1} holds by definition. Step 4 is once more \eqref{eq:extension-condition-aux2}, step 5 follows as $\xi^\delta(\alpha^j_{-n})  \hat u'$ is equally an element of the defining basis, and finally step 6 amounts to inserting the definition of $\hat u$.

\item For $m>-n$, the argument is similar to the one above: 
Rewrite 
\be
\xi^\delta(\alpha^i_{m}) \hat u 
= 
-\xi^\delta(\alpha^j_{-n})\xi^\delta(\alpha^i_{m})  \hat u'
+[\xi^\delta(\alpha^i_{m}),\xi^\delta(\alpha^j_{-n})] \hat u' \ .
\ee
Since $\hat u'$ and $\xi^\delta(\alpha^i_{m})  \hat u'$ have grade $\le N$, we can use the induction hypothesis to carry out a calculation like the one above to establish \eqref{eq:extension-condition-aux1}  -- we omit the details.
\end{itemize}

\section{The vertex operator corresponding to $H$}\label{sec:OPEs}

This appendix is complementary to the discussion in Section \ref{sec:Vd-inside-lattice}. There, an explicit embedding of the even part of the symplectic fermion VOA $\VOA(\h)_\mathrm{ev}$ into the lattice VOA for the even self-dual lattice $D^+_{d/2}$ was constructed. Here, conversely, we identify a rank-$\frac d2$ Heisenberg sub-VOA in each extension of $\VOA(\h)_\mathrm{ev}$ as classified in Theorem \ref{thm:classify-extensions}. As that theorem relies on the assumptions in Section \ref{sec:two-assumptions}, the calculations in this appendix lend support to these assumptions (only Remark \ref{rem:Heis-sub-VOA} below depends on these assumptions).

\medskip

Recall that $H = L \oplus T$ where $L$ is generated by $S(\f) \subset S^{\frac d2}(\h)$ (Lemma \ref{lem:kernel-is-Lag}), and that the multiplication on $H$ is determined by $m \in S(\f)$ and $1_H \in S^d$ (Section \ref{sec:ex-unique}).
Pick a complement $\f^*$ of $\f$ such that $\h = \f^* \oplus \f$, and such that the symplectic pairing on $\h$ satisfies $(\varphi,u) = \varphi(u)$ for all $\varphi \in \f^*$ and $u \in \f$. 
We obtain an isomorphism $\widetilde{(-)} : \f \to L \cap S^{d-1}$ by requiring that for all $\varphi \in \f^*$ and $a \in \f$,
\be
	\varphi \,\widetilde a = (\varphi,a) \, 1_H \ .
\ee
Denote the ground state of $\widehat T$ by $\tau$ (i.e.\ $\tau=1\in T$). 
Pick a Lagrangian subspace $\f \subset \h$ and set $H = H(\f)$. Denote by 
\be
	V := V(\mu,-) : \Rb_{>0} \times (\widehat H \otimes \widehat H) \to \overline{\widehat H}
\ee 
the vertex operator determined by multiplication map $\mu : H * H \to H$ as in \eqref{eq:SF-tensor-via-rep-obj}. 
From \eqref{eq:Pgs-V} and the definition of $\mu$ in Section \ref{sec:ex-unique} we know that, for $u,v \in L$ 
\begin{align}
V(x) (u \otimes v) &= u \Lmul v + O(x)\ ,
\nonumber\\
V(x) (u \otimes \tau) &= 1_H^*(u) \, \tau  + O(x^{\frac12})\ ,
\nonumber\\
V(x) (\tau \otimes v) &= 1_H^*(v) \, \tau  + O(x^{\frac12})\ ,
\nonumber\\
V(x) (\tau \otimes \tau) &= x^{\frac d8} \,( m  + O(x^{\frac12}) ) \ .
\label{eq:OPEs-leading-order}
\end{align}
Here we used that $C$ acts as zero on $L \otimes L$ and $\hat C$ acts as zero on $L$ (see Section \ref{sec:ex-unique}).
We can now use Definition \ref{def:vertex-op} and Proposition \ref{prop:V-on-descendents} to compute the leading part of $V(x)$. We are particularly interested in the case that the arguments are ground states and untwisted states of weight one.

\begin{lemma}\label{lem:V-expansions}
We have, for all $a,b\in \h$ and $u,v \in L$
\begin{align}
V(x) \big( (a_{-1} u) \otimes v \big)
=~& x^{-1} \, (-1)^{|u|} \, u \Lmul (a.v) ~+~ \text{reg.}
\nonumber \\
V(x) \big( u \otimes (b_{-1} v) \big)
=~& x^{-1} \, (-1)^{|u|+1} \, (b.u) \Lmul v ~+~ \text{reg.}
\nonumber \\
V(x)  \big( (a_{-1} u) \otimes (b_{-1} v) \big)
=~& 
x^{-2}\, \big\{  (a,b) (-1)^{|u|} u \Lmul v \,+\, (b.u) \Lmul (a.v) \big\}
\nonumber \\
& + x^{-1} \Big\{ -(-1)^{|u|} a_{-1}((b.u) \Lmul v) - b_{-1} (u \Lmul (a.v))
\nonumber \\
&
\hspace{4em}
+ \sum_{j=1}^d \beta^j_{-1} ((\alpha^j b.u) \Lmul (a.v)) 
\nonumber \\
&
\hspace{4em}
+ (-1)^{|u|} (a,b) \sum_{j=1}^d \beta^j_{-1} ((\alpha^j.u) \Lmul v) \Big\}
 ~+~ \text{reg.}
\nonumber \\
V(x) \big(  (a_{-1} u) \otimes \tau \big)
=~& 
-\tfrac12 \,x^{-1} \,1_H^*(a.u) \, \tau 
\nonumber\\
& +  x^{-\frac12} \Big\{1_H^*(u) \, a_{-\frac12}\tau
- \sum_{j=1}^d 1_H^*(\alpha^ja.u) \, \beta^j_{-\frac12} \tau \Big\}
~+~ \text{reg.}
\nonumber \\
V(x) \big(  \tau \otimes (b_{-1} v) \big)
=~& 
\tfrac12 \, x^{-1} \, 1_H^*(b.v) \, \tau 
\nonumber\\
& - i x^{-\frac12} \Big\{1_H^*(v) \, b_{-\frac12}\tau- \sum_{j=1}^d 1_H^*(\alpha^jb.v) \, \beta^j_{-\frac12} \tau \Big\}
~+~ \text{reg.}
\label{eq:V(x)-leading-terms}
\end{align}
where ``reg.'' stands for terms multiplied by $x^k$ with $k \ge 0$.
\end{lemma}

\begin{proof}
Pick a basis $\alpha^i$, $i=1,\dots,d$ of $\h$ and let $\beta^i$ be the basis dual with respect to $(-,-)$ (and not $(-,-)_{\SF}$, cf.\ \eqref{eq:two-pairings}), i.e.\ $(\alpha^i,\beta^j)=\delta_{ij}$.
We will need the first three OPEs in \eqref{eq:OPEs-leading-order} to subleading order. For $V(x) (u \otimes v)$ we make the general ansatz
\be
	V(x) (u \otimes v) = u \Lmul v + x \, \sum_{i=1}^d \beta^i_{-1} w_i + O(x^2) \ ,
\ee
where $w_i \in L$, and then we act with $\alpha^j_1$ on both sides. Using \eqref{eq:mode-past-untwisted}, the left hand side gives
$\alpha^j_1 \, V(x)  (u \otimes v)
	= x \, V(x) (\alpha^j.u \otimes v) = x \,  (\alpha^j.u) \Lmul v + O(x^2)$.
Using \eqref{eq:SF-liebracket} and that $u \Lmul v$
is a ground state, the right hand side simplifies as
$x \, \sum_i \alpha^j_1\beta^i_{-1} w_i + O(x^2) = x\, w_j + O(x^2)$. Comparing coefficients of $x$ we conclude
\be\label{eq:V(x)uxv-2nd-order}
	V(x) (u \otimes v) = u \Lmul v + x \, \sum_{j=1}^d \beta^j_{-1} \big( (\alpha^j.u) \Lmul v \big) + O(x^2) \ .
\ee
A similar calculation for the other two cases yields
\begin{align}
V(x)(u \otimes \tau) &= 
1_H^*(u) \, \tau
+
x^{\frac12}\, 2 \sum_{j=1}^d 1_H^*(\alpha^j.u) \, \beta^j_{-\frac12} \tau + O(x) \ .
\nonumber\\
V(x)(\tau \otimes v) &= 
1_H^*(v) \, \tau
+
x^{\frac12}\,2i \sum_{j=1}^d  1_H^*(\alpha^j.v) \, \beta^j_{-\frac12} \tau + O(x) \ .
\end{align}
Next we compute $V(x) \big( u \otimes (b_{-1} v) \big)$, which we will need to constant order.
We use \eqref{eq:mode-past-untwisted} in the form
$b_{-1} V(x)(u \otimes v) = V(x)\big( x^{-1} (b.u) \otimes v + (-1)^{|u|} u \otimes (b_{-1}v) \big)$ and insert the expansion computed in \eqref{eq:V(x)uxv-2nd-order}. This gives
\begin{align}
	V(x) \big( u \otimes (b_{-1} v) \big)
	=~&
	x^{-1} \, (-1)^{|u|+1} (b.u) \Lmul v
	\nonumber \\
	&+
	(-1)^{|u|} b_{-1}(u \Lmul v)
	-
	(-1)^{|u|} \sum_{j=1}^d \beta^j_{-1} \big( (\alpha^jb.u) \Lmul v \big)
	+ O(x) 
\label{eq:V(x)uxbv-1st-order}
\end{align}
and thus recovers the corresponding expression in \eqref{eq:V(x)-leading-terms}.
A similar calculation recovers the expansion of $V(x) \big(  \tau \otimes (b_{-1} v) \big)$ stated in \eqref{eq:V(x)-leading-terms}.

Finally we consider the expansions involving $a_{-1}u$ in the first argument. Let us give the details for $V(x) \big( (a_{-1} u) \otimes v \big)$ and $V(x) \big( (a_{-1} u) \otimes (b_{-1}v) \big)$. We use \eqref{eq:mode-at-x_untw0} for $m=-1$ acting on $u \otimes \psi$, where $\psi$ is either $v$ or $b_{-1}v$:
\begin{align}
V(x)\big( (a_{-1}u) \otimes \psi \big)
=~&
a_{-1} V(x)(u \otimes \psi) + \dots
\nonumber \\
&+ (-1)^{|u|} x^{-1} V(x)\big(u \otimes (a_0\psi)\big)
+ (-1)^{|u|} x^{-2} V(x)\big(u \otimes (a_1\psi)\big)
 \ ,
\end{align}
where ``$\dots$'' stands for terms in the first sum on the left hand side of \eqref{eq:mode-at-x_untw0} with $x^k$ for $k>0$. Substituting \eqref{eq:V(x)uxv-2nd-order} and \eqref{eq:V(x)uxbv-1st-order} produces the expansions in \eqref{eq:V(x)-leading-terms}. The expansion of $V(x) \big(  (a_{-1} u) \otimes \tau \big)$ is obtained analogously from \eqref{eq:mode-at-x_tw0}.
\end{proof}

The above OPEs contain square roots of $x$, but they have single valued extensions to $\Cb^\times$ if one restricts to the even subspace.

\medskip

Next we will look for OPEs in $(\widehat H)_{\mathrm{ev}}$ that match those of a rank-$\frac d2$ Heisenberg VOA.
Pick a basis $\{f^i\,|\,i=1,\dots,\tfrac d2\}$ of $\f$ and let $\{f^{i*}\}$ be the dual basis of $\f^*$. Then $(f^{i*},f^j) = \delta_{i,j} = - (f^j,f^{i*})$. Define, for $i=1,\dots,\tfrac d2$,
\be
	H^i := f^{i*}_{-1}\widetilde{f^{i}} \in (\widehat L)_\mathrm{ev}
	\quad , \quad
	H^i(x) \,:=\, V(x) \circ ( H^i \otimes \id) 
	~:~
	(\widehat H)_{\mathrm{ev}} \to \overline{(\widehat H)}_{\mathrm{ev}} 
\ee
From Lemma \ref{lem:V-expansions} we get
\be\label{eq:Heis-OPE-in-Hev}
	H^i(x)H^j
	~=~
	\delta_{i,j} \, x^{-2} \,  1_H ~+~ \text{reg.} \ .
\ee

\begin{remark}\label{rem:Heis-sub-VOA}
If we assume Conjecture \ref{conj:braided-equiv}, by Theorem \ref{thm:HKL}, $(\widehat H)_{\mathrm{ev}}$ is a VOA. The OPE \eqref{eq:Heis-OPE-in-Hev} then shows that
$(\widehat H)_\mathrm{ev}$ (in fact $(\widehat L)_\mathrm{ev}$)
contains a $\frac d2$-dimensional Heisenberg VOA. 
\end{remark}

We can expand $H^i(x)$ into modes as $H^i(x) = \sum_{m \in \Zb} H^i_m x^{-m-1}$. To see how these modes act on $(\widehat L)_\mathrm{ev}$, we need some more notation. 
For $A,B,C \in \SF_0$ and $f \in \SF(A * B,C)$ let $Q_f : A \otimes \widehat B \to \overline{\widehat{C}}$ be defined by the properties \cite[Lem.\,3.5]{Runkel:2012cf}
\begin{itemize}
\item for all $a\in A$, $b \in B$:
	$Q_f(a \otimes b) = f(a \otimes b)\ $,
\item for all $h \in \h$, $m \in \Zb$:
	$h_m \,Q_f = \delta_{m,0}\, Q_f\circ (h \otimes \id) + Q_f \circ (\id \otimes h_m)\ $.
\end{itemize}
We will use this in the case $f = \mu|_{L \otimes L} : L \otimes L \to L$. 

Next, for all $h,k \in \h$ and $m,n \in \Zb$ let the normal ordered product be
\be
	:h_mk_n: ~=~ \begin{cases} h_m k_n &; m \le n \\
	-k_n h_m &; m>n
	\end{cases}\quad .
\ee
Note that by \eqref{eq:SF-liebracket}, the two alternatives can be different only for $m=-n$.

\begin{lemma}\label{lem:mode-action-on-Lhat}
Let $h \in \h$, $f \in \f$ and $v \in (\widehat L)_\mathrm{ev}$. Then
\be 
	V(x) (h_{-1}\widetilde{f} \otimes v) = \sum_{k \in \Zb} x^{-k-1} X_k v \ ,
\ee
where 
\be
	X_k v = - Q_\mu(\widetilde f \otimes h_k v) - \sum_{m \in \Zb \setminus \{0\}} \tfrac{1}{m} :h_{k-m}f_m:\,v \quad .
\ee
\end{lemma}

\begin{proof}
By Proposition \ref{prop:V-on-descendents},
\be\label{eq:mode-action-aux1}
	V(x)(h_{-1}\widetilde{f} \otimes v)
	= \sum_{l=0}^\infty x^l \,h_{-l-1} V(x) (\widetilde f \otimes v)
	- \sum_{l=0}^\infty x^{-l-1}\, V(x) (\widetilde f \otimes h_l v) \ .
\ee
Next we substitute the expression for $V(x)$ in terms of normal ordered exponentials given in \cite[Lem.\,3.6]{Runkel:2012cf} to find, for all $w \in \widehat L$,
\be\label{eq:mode-action-aux2}
	V(x)(\widetilde f \otimes w)
	= \sum_{m \neq 0} \tfrac{x^m}m f_{-m}w + Q_\mu(\widetilde f \otimes w) \ .
\ee
In computing this, we used that only the first two terms in the expansion of each exponential contribute, and that $E_0(x) = \id$ as $f_0$ acts as zero on $\widehat L$ (see \cite{Runkel:2012cf} for the definition of $E_0(x)$).

Combining \eqref{eq:mode-action-aux1} and \eqref{eq:mode-action-aux2} gives the statement of the lemma.
\end{proof}

The lemma shows that the modes $H^i_k$ act on $v \in \widehat L$ as
\be\label{eq:Hik-in-terms-of-f}
	H^i_k
	=- Q_\mu(\widetilde{f^i} \otimes f^{i*}_k v) - \sum_{m \in \Zb \setminus \{0\}} \tfrac{1}{m} :f^{i*}_{k-m}f^i_m:\,v \quad .	
\ee

The symplectic fermion stress tensor \eqref{eq:SF-stress-tensor} can be expressed in terms of the modes $H^i_k$ as follows:

\begin{lemma}\label{lem:Vir-el-via-Heis}
We have $L_{-2} 1_H = \frac12 \sum_{i=1}^{d/2} (H^i_{-1} H^i_{-1} - H^i_{-2}) 1_H$. 
\end{lemma}

\begin{proof}
In the basis consisting of $f^i$ and $f^{*i}$, \eqref{eq:virasoro-action} gives $L_{-2}1_H = \sum_{i=1}^{d/2} f^i_{-1} f^{i*}_{-1} 1_H$. 
From \eqref{eq:Hik-in-terms-of-f} we read off $H_{-1}^i 1_H = f^{i*}_{-1} \widetilde{f^i}$ and $H^i_{-2}1_H = f^{i*}_{-2} \widetilde{f^i} - f^{i}_{-1}f^{i*}_{-1}1_H$. Again using \eqref{eq:Hik-in-terms-of-f}, one computes $H_{-1}^i H_{-1}^i 1_H = f^{i*}_{-2} \widetilde{f^i} + f^{i}_{-1}f^{i*}_{-1}1_H$. This shows the statement of the lemma.
\end{proof}

The expression for the stress tensor obtained in the above lemma agrees with that in Theorem \ref{thm:symp-ferm-in-lattice}, as it should. 

\medskip

One can prove a result similar to Lemma \ref{lem:mode-action-on-Lhat} for the action of $H^i_k$ on $\widehat T$. This allows one to determine the spectrum of the zero modes $H^i_0$ on $(\widehat H)_\mathrm{ev}$. The result is the lattice $D^+_{d/2}$, as expected by Corollary \ref{cor:symp-even-in-lattice} -- we omit the details.

\newcommand\arxiv[2]      {\href{http://arXiv.org/abs/#1}{#2}}
\newcommand\doi[2]        {\href{http://dx.doi.org/#1}{#2}}
\newcommand\httpurl[2]    {\href{http://#1}{#2}}

\end{document}